\documentclass[a4paper]{article}
\usepackage[margin=30mm]{geometry}
\usepackage{amsmath}
\usepackage{amssymb}
\usepackage{amsthm}
\usepackage{titlesec}
\usepackage{titlefoot}
\titleformat*{\section}{\normalsize\bfseries}
\def\R{\mathbb{R}}

\def\e{{\varepsilon}}

\def\p{\partial}
\def\d{\displaystyle}
\newtheorem{thm}{Theorem}[section]
\newtheorem{lem}[thm]{Lemma}
\newtheorem{cor}[thm]{Corollary}
\newtheorem{prop}[thm]{Proposition}
\newtheorem{rem}[thm]{Remark}

\makeatletter
\@addtoreset{equation}{section}

\makeatother

\begin{document}

\title{\large{\bf{Large Time Behavior of Solutions to a Nonlinear Hyperbolic \\Relaxation System with Slowly Decaying Data}}}
\author{Ikki Fukuda\\ [.7em]
Department of Mathematics, Hokkaido University}
\date{}
\maketitle

\footnote[0]{2010 Mathematics Subject Classification: 35B40, 35L70, 35L15.}

\begin{abstract}
We consider the large time asymptotic behavior of the global solutions to the initial value problem for the nonlinear damped wave equation with slowly decaying initial data. When the initial data decay fast enough, it is known that the solution to this problem converges to the self-similar solution to the Burgers equation called a nonlinear diffusion wave and its optimal asymptotic rate is obtained. In this paper, we focus on the case that the initial data decay more slowly than previous works and derive the corresponding asymptotic profiles. Moreover, we investigate how the change of the decay rate of the initial values affect its asymptotic rate.
\end{abstract}

\smallskip
{\bf Keywords:} nonlinear damped wave equation, hyperbolic relaxation system,  asymptotic profile, \\
\indent second asymptotic profile, optimal decay estimate, slowly decaying data. 

\section{Introduction}

In this paper, we consider the large time behavior of the global solutions to the initial value problem for the following system:
\begin{align}
\label{1-1}
\begin{split}
&u_{t}+v_{x}=0, \ \  v_{t}+u_{x}=f(u)-v, \ x\in \R,\ \ t>0, \\
&u(x, 0)=u_{0}(x), \ v(x, 0)=v_{0}(x), \ x\in \R, 
\end{split}
\end{align}
where $f: \R \to \R$ is a given smooth function. This system is a typical example of hyperbolic system of conservation laws with relaxation called Jin-Xin model, which arises as mathematical models in several physical phenomena, e.g. non-equilibrium gas dynamics, magnetohydrodynamics and viscoelasticity (see e.g. \cite{JX95, W74}). 

If we delete $v$ from \eqref{1-1}, we obtain the following damped wave equation with a nonlinear convection term:
\begin{align}
\label{1-2}
\begin{split}
&u_{tt}-u_{xx}+u_{t}+(f(u))_{x}=0, \ x\in \R,\ \ t>0, \\
&u(x, 0)=u_{0}(x), \ u_{t}(x, 0)=u_{1}(x),\ x\in \R, 
\end{split}
\end{align}
where the initial data $u_{1}(x)=-\p_{x}v_{0}(x)$. In the present paper, we consider \eqref{1-2} with the flux function $f(u)\equiv au+\frac{b}{2}u^{2}+\frac{c}{3!}u^{3}$, where $|a|<1$, $b\neq0$ and $c\in \R$. In addition, we assume that 
\begin{align}
\label{1-3}
\begin{split}
&\exists \alpha>1, \ \ \exists C>0 \ \ s.t. \ \ |u_{0}(x)|\le C(1+|x|)^{-\alpha}, \ \ x\in \R, \\
&\exists \beta>1, \ \ \exists C>0 \ \ s.t. \ \ |u_{1}(x)|\le C(1+|x|)^{-\beta}, \ \ x\in \R.  
\end{split}
\end{align}
The purpose of our study is to obtain an asymptotic profile of the solution $u(x, t)$ and to examine the optimality of its asymptotic rate to the asymptotic function. 

First of all, we recall known results about the asymptotic behavior of the solutions to \eqref{1-2}. Orive and Zuazua~\cite{OZ06} studied the global existence and the asymptotic behavior of the solution of \eqref{1-2} with $a=0$ when $u_{0}\in H^{1}(\R) \cap L^{1}(\R)$ and $u_{1}\in L^{2}(\R) \cap L^{1}(\R)$. Moreover, Ueda and Kawashima~\cite{UK07} constructed the solution to \eqref{1-2}, provided the initial data $u_{0}\in W^{1, p}(\R) \cap L^{1}(\R)$ and $u_{1}\in L^{p}(\R) \cap L^{1}(\R)$ for $1\le p \le \infty$, and they showed that the solution of \eqref{1-2} converges to a nonlinear diffusion wave defined by
\begin{equation}
\label{1-4}
\chi(x, t)\equiv \frac{1}{\sqrt{1+t}} \chi_{*} \biggl(\frac{x-a(1+t)}{\sqrt{1+t}}\biggl), \ x\in \R,\ \ t>0,
\end{equation}
where
\begin{equation}
\label{1-5}
\chi_{*}(x)\equiv \frac{\sqrt{\mu}}{b}\frac{(e^{\frac{bM}{2\mu}}-1)e^{-\frac{x^{2}}{4\mu}}}{\sqrt{\pi}+(e^{\frac{bM}{2\mu}}-1)\int_{x/\sqrt{4\mu}}^{\infty}e^{-y^{2}}dy}, \ M \equiv \int_{\R}(u_{0}(x)+u_{1}(x))dx, \ \mu \equiv 1-a^{2}.
\end{equation}
More precisely, if $u_{0}\in W^{1, p}(\R) \cap L^{1}_{1}(\R)$, $u_{1}\in L^{p}(\R) \cap L^{1}_{1}(\R)$ and $\|u_{0}\|_{W^{1, p}}+\|u_{0}\|_{L^{1}}+\|u_{1}\|_{L^{p}}+\|u_{1}\|_{L^{1}}$ is sufficiently small, then, for any $\e>0$, we have 
\begin{equation}
\label{1-7}
\|\p_{x}^{l}(u(\cdot, t)-\chi(\cdot, t)) \|_{L^{p}} \le C(1+t)^{-1+\frac{1}{2p}-\frac{l}{2}+\e}, \ \ t\ge0, \ \ l=0, 1.
\end{equation}
Here the weighted Lebesgue space $L_{1}^{1}(\R)$ is defined by 
\begin{equation*}
L^{1}_{1}(\R) \equiv \biggl\{ f \in L^{1}(\R); \ \|f\|_{L^{1}_{1}} \equiv \int_{\R}|f(x)|(1+|x|)dx<\infty \biggl\}.
\end{equation*}
Also, we note that $\chi(x, t)$ is a solution of the following Burgers equation:
\begin{equation}
\label{1-6}
\chi_{t}+\left(a\chi+\frac{b}{2}\chi^{2}\right)_{x}=\mu \chi_{xx}, \ \ \int_{\R}\chi(x, 0)dx=M.  
\end{equation}

Moreover, the optimality of the asymptotic rate to the nonlinear diffusion wave was obtained by Kato and Ueda~\cite{KU17} by constructing the second asymptotic profile which is the leading term of $u-\chi$. Indeed, if $u_{0}\in W^{s, p}(\R) \cap W^{2, 1}(\R) \cap L^{1}_{1}(\R)$, $u_{1}\in W^{s-1, p}(\R) \cap W^{1, 1}(\R) \cap L^{1}_{1}(\R)$ for $s\ge2$ and $1\le p\le \infty$, and $\|u_{0}\|_{W^{s, p}}+\|u_{0}\|_{W^{2, 1}}+\|u_{1}\|_{W^{s-1, p}}+\|u_{1}\|_{W^{1, 1}}$ is sufficiently small, then we have
\begin{equation}
\label{1-8}
\|\p_{x}^{l}(u(\cdot, t)-\chi(\cdot, t)-V(\cdot, t))\|_{L^{p}} \le C(1+t)^{-1+\frac{1}{2p}-\frac{l}{2}}, \ \ t\ge1
\end{equation}
for $0\le l\le s-2$, where
\begin{equation}
\label{1-9}
V(x, t)\equiv -\kappa dV_{*}\biggl(\frac{x-a(1+t)}{\sqrt{1+t}}\biggl)(1+t)^{-1}\log(1+t), 
\end{equation}
and
\begin{align}
\label{1-10}
&V_{*}(x)\equiv \frac{1}{\sqrt{4\pi \mu}}\frac{d}{dx}(\eta_{*}(x)e^{-\frac{x^{2}}{4\mu}}), \ \ \eta_{*}(x)\equiv \exp \biggl(\frac{b}{2\mu}\int_{-\infty}^{x}\chi_{*}(y)dy\biggl), \\
\label{1-11}
&d\equiv \int_{\R}(\eta_{*}(y))^{-1}(\chi_{*}(y))^{3}dy, \ \ \kappa \equiv \frac{ab^{2}}{4\mu}+\frac{c}{3!}. 
\end{align}
From \eqref{1-8}, the triangle inequality and \eqref{1-9}, one can obtain the following optimal decay estimate:
\begin{equation}
\label{1-12}
\|\p_{x}^{l}(u(\cdot, t)-\chi(\cdot, t))\|_{L^{p}}=(\tilde{C}+o(1))(1+t)^{-1+\frac{1}{2p}-\frac{l}{2}}\log(1+t), \ \ 0\le l\le s-2
\end{equation}
as $t\to \infty$, where $\tilde{C}\equiv |\kappa d|\|\p_{x}^{l}V_{*}\|_{L^{p}}$. Therefore, we see that the solution $u(x, t)$ to \eqref{1-2} tends to the nonlinear diffusion wave $\chi(x, t)$ at the rate of $t^{-1+\frac{1}{2p}}\log t$ in the $L^{p}$-sense if $M \neq0$ and $\kappa \neq0$, i.e., we cannot take $\e=0$ in \eqref{1-7}. The similar estimates for \eqref{1-7} and \eqref{1-8} are obtained for Burgers type equations such as generalized Burgers equation, KdV-Burgers equation and BBM-Burgers equation (cf. \cite{F19-1, F19-2, HKN07, HN06, KP05, Ka07, MN04}).

The above results \cite{UK07, KU17} are corresponding to the case where the decay rate of the initial data $u_{0}$ and $u_{1}$ are rapid because $u_{0}, u_{1}\in L^{1}_{1}(\R)$ are realized when $\alpha, \beta>2$ in \eqref{1-3}. However for \eqref{1-2} in the case of $1<\alpha \le2$ or $1<\beta \le 2$ in \eqref{1-3}, it is not known that the optimal asymptotic rate to the nonlinear diffusion wave, as far as we know. On the other hand, it is studied that the asymptotic profile for the solution to the damped wave equation with power type nonlinearity for slowly decaying data in the supercritical case. Actually, Narazaki and Nishihara~\cite{NN08} studied the following equation when the initial data are not in $L^{1}(\R)$:
\begin{align}
\label{1-13}
\begin{split}
&u_{tt}-u_{xx}+u_{t}=|u|^{p-1}u, \ x\in \R,\ \ t>0, \\
&u(x, 0)=u_{0}(x), \ u_{t}(x, 0)=u_{1}(x),\ x\in \R.
\end{split}
\end{align}
They assumed that the initial data satisfies the condition \eqref{1-3} with $\alpha=\beta=:k$ and $0<k \le1$, and showed that if $p>1+2/k$ and the data $u_{0}\in B^{1, k}(\R)$, $u_{1}\in B^{0, k}(\R)$ are small, then the asymptotic profile is given by 
\begin{equation}
\label{1-14}
\Psi(x, t)=c_{k}\int_{\R}\frac{1}{\sqrt{4\pi t}}e^{-\frac{(x-y)^{2}}{4t}}(1+|y|)^{-k}dy, 
\end{equation}
provided that the data satisfies $\lim_{|x|\to \infty}(1+|x|)^{k}(u_{0}+u_{1})(x)=c_{k}$. Here, we set 
\[B^{m, k}(\R)\equiv \{f\in C^{m}(\R); (1+|x|)^{k}|\p_{x}^{l}f|\in L^{\infty}(\R), \ 0\le l \le m\}.\] 
More precisely, they proved  
\begin{align}
\label{1-15}
\begin{split}
\lim_{t\to \infty} a_{k}(t)\|u(\cdot, t)-\Psi(\cdot, t)\|_{L^{\infty}}=0, \ a_{k}(t)=\begin{cases}
(1+t)^{k/2}, &0<k<1, \\
\frac{(1+t)^{1/2}}{\log(1+t)}, &k=1.\\
\end{cases}
\end{split}
\end{align}
Moreover in~\cite{NN08}, the damped wave equation \eqref{1-13} in two and three space dimensional cases were also studied. For the related results concerning \eqref{1-13}, we also refer to \cite{IIOW19, IIW17}. However, as we mentioned in the above, the asymptotic profile of the solutions to \eqref{1-2} with slowly decaying data is not well known even if the data are in $L^{1}(\R)$. For this reason, we would like to analyze the asymptotic behavior of the solution to \eqref{1-2} in the case of $1<\alpha \le2$ or $1<\beta \le2$ in \eqref{1-3}. \\

Now, we state our first main result which generalizes the result given in~\cite{UK07}: 
\begin{thm}\label{main1}
Assume the condition \eqref{1-3} holds with $1<\min\{\alpha, \beta\} \le 2$. Let $s$ be a positive integer and $1\le p\le \infty$. Suppose that $u_{0}\in W^{s, p}(\R)$, $u_{1}\in W^{s-1,p}(\R)$ and $\|u_{0}\|_{W^{s, p}}+\|u_{0}\|_{L^{1}}+\|u_{1}\|_{W^{s-1, p}}+\|u_{1}\|_{L^{1}}$ is sufficiently small. Then \eqref{1-2} has a unique global solution $u(x, t)$ with
\begin{equation*}
u\in \begin{cases}
\bigcap_{k=0}^{\sigma}C^{k}([0, \infty); W^{s-k, p})\cap C([0, \infty); L^{1}), &1\le p <\infty,\\
\bigcap_{k=0}^{\sigma}W^{k, \infty}(0, \infty; W^{s-k, \infty})\cap C([0, \infty); L^{1}), &p=\infty,\end{cases}
\end{equation*}
where $\sigma=\min\{2, s\}$. Moreover, for any $\e>0$, the estimate
\begin{align}
\label{1-16}
\|u(\cdot, t)-\chi(\cdot, t)\|_{L^{q}}\le C\begin{cases}
(1+t)^{-\frac{\min\{\alpha, \beta\}}{2}+\frac{1}{2q}}, &t\ge0, \ 1<\min\{\alpha, \beta\}<2,\\
(1+t)^{-1+\frac{1}{2q}+\e}, &t\ge0, \ \min\{\alpha, \beta\}=2
\end{cases}
\end{align}
holds for any $q$ with $1\le q\le \infty$, and the estimate 
\begin{align}
\label{1-17}
\|\p_{t}^{k}\p_{x}^{l}(u(\cdot, t)-\chi(\cdot, t))\|_{L^{p}}\le C\begin{cases}
(1+t)^{-\frac{\min\{\alpha, \beta\}}{2}+\frac{1}{2p}-\frac{k+l}{2}}, &t\ge0, \ 1<\min\{\alpha, \beta\}<2,\\
(1+t)^{-1+\frac{1}{2p}-\frac{k+l}{2}+\e}, &t\ge0, \ \min\{\alpha, \beta\}=2
\end{cases}
\end{align}
holds for $0\le k\le2$ and $l\ge 0$ with $0\le k+l\le s$, where $\chi(x, t)$ is defined by \eqref{1-4}.
\end{thm}
Furthermore, we can show that the above asymptotic rate given in \eqref{1-16} is optimal with respect to the time decaying order in the $L^{\infty}$-sense by constructing the second asymptotic profile for the solution to \eqref{1-2}. To state such a result, we define the following functions
\begin{equation}
\label{1-21}
Z(x, t)\equiv \int_{\R}\frac{c_{\alpha, \beta}(y)\p_{x}(G_{0}(x-y, t)\eta(x, t))}{(1+|y|)^{\min\{\alpha, \beta\}-1}}dy, \ \ c_{\alpha, \beta}(y)\equiv\begin{cases}
c_{\alpha, \beta}^{+}, &y\ge0, \\
c_{\alpha, \beta}^{-}, &y<0
\end{cases}
\end{equation}
and 
\begin{equation}
\label{1-22}
G_{0}(x, t)\equiv \frac{1}{\sqrt{4\pi \mu t}}e^{-\frac{(x-at)^{2}}{4\mu t}}, \ \ \eta(x, t)\equiv \eta_{*}\biggl(\frac{x-a(1+t)}{\sqrt{1+t}}\biggl)=\exp\biggl(\frac{b}{2\mu}\int_{-\infty}^{x}\chi(y, t)dy\biggl)
\end{equation}
with $\eta_{*}(x)$ being defined by \eqref{1-10}. Then, we have the following result: 
\begin{thm}\label{main2}
Assume the condition \eqref{1-3} holds with $1<\min\{\alpha, \beta\} \le 2$, $u_{0}\in H^{2}(\R)\cap W^{2,1}(\R)$ and $u_{1}\in H^{1}(\R)\cap W^{1, 1}(\R)$. Suppose that $\|u_{0}\|_{H^{2}}+\|u_{0}\|_{W^{2,1}}+\|u_{1}\|_{H^{1}}+\|u_{1}\|_{W^{1,1}}$ is sufficiently small. We set $\chi_{0}(x)\equiv \chi(x, 0)$, $\eta_{0}(x)\equiv \eta(x, 0)$ and 
\begin{equation}
\label{1-18}
z_{0}(x)\equiv \eta_{0}(x)^{-1}\int_{-\infty}^{x}(u_{0}(y)+u_{1}(y)-\chi_{0}(y))dy.
\end{equation} 
If there exists $\lim_{x\to \pm \infty}(1+|x|)^{\min\{\alpha, \beta\}-1}z_{0}(x)\equiv c_{\alpha, \beta}^{\pm}$, then the solution to \eqref{1-2} satisfies
\begin{align}
\label{1-19}
&\lim_{t\to \infty}(1+t)^{\frac{\min\{\alpha, \beta\}}{2}}\|u(\cdot, t)-\chi(\cdot, t)-Z(\cdot, t)\|_{L^{\infty}}=0, \ 1<\min\{\alpha, \beta\}<2, \\
\label{1-20}
&\lim_{t\to \infty}\frac{(1+t)}{\log(1+t)}\|u(\cdot, t)-\chi(\cdot, t)-Z(\cdot, t)-V(\cdot, t)\|_{L^{\infty}}=0, \ \min\{\alpha, \beta\}=2,  
\end{align}
where $\chi(x, t)$ and $V(x, t)$ are defined by \eqref{1-4} and \eqref{1-9}, respectively, while $Z(x, t)$ and $\eta(x, t)$ are defined by \eqref{1-21} and \eqref{1-22}, respectively. Moreover, if $M \neq0$, there exist $\nu_{0}>0$ and $\nu_{1}>0$ independent of $x$ and $t$ such that 
\begin{align}
\label{1-23}
\begin{split}
\|Z(\cdot, t)\|_{L^{\infty}} \begin{cases}
\le C\max\{|c_{\alpha, \beta}^{+}|, |c_{\alpha, \beta}^{-}|\}(1+t)^{-\frac{\min\{\alpha, \beta\}}{2}},  \\[.2em]
\ge \nu_{0}|\tilde{\nu_{0}}|(1+t)^{-\frac{\min\{\alpha, \beta\}}{2}} &
\end{cases}
\end{split}
\end{align}
holds for sufficiently large $t$ with $1<\min\{\alpha, \beta\}<2$ and 
\begin{align}
\label{1-24}
\begin{split}
\|Z(\cdot, t)+V(\cdot, t)\|_{L^{\infty}} \begin{cases}
\le C(\max\{|c_{\alpha, \beta}^{+}|, |c_{\alpha, \beta}^{-}|\}+|\kappa d|\|V_{*}(\cdot)\|_{L^{\infty}})(1+t)^{-1}\log(1+t), & \\
\ge \nu_{1}|\tilde{\nu_{1}}|(1+t)^{-1}\log(1+t) &
\end{cases}
\end{split}
\end{align}
holds for sufficiently large $t$ with $\min\{\alpha, \beta\}=2$, where 
\begin{align}
\label{1-25}
\begin{split}
\tilde{\nu_{0}} &\equiv \sqrt{\mu}(c_{\alpha, \beta}^{+}-c_{\alpha, \beta}^{-})\Gamma\biggl(\frac{3-\min\{\alpha, \beta\}}{2}\biggl)+\frac{b\chi_{*}(0)(c_{\alpha, \beta}^{+}+c_{\alpha, \beta}^{-})}{2-\min\{\alpha, \beta\}}\Gamma\biggl(2-\frac{\min\{\alpha, \beta\}}{2}\biggl), \\
\tilde{\nu_{1}} &\equiv \frac{c_{\alpha, \beta}^{+}+c_{\alpha, \beta}^{-}}{2}-\kappa d, \ \ \Gamma (s)\equiv \int_{0}^{\infty}e^{-x}x^{s-1}dx, \ \ s>0,
\end{split}
\end{align}
while $M$, $d$ and $\kappa$ are defined by \eqref{1-5} and \eqref{1-11}, respectively. 
\end{thm}
\noindent 
By virtue of Theorem 1.2, the optimality of the estimate \eqref{1-16} can be examined from \eqref{1-19}, \eqref{1-20}, \eqref{1-23} and \eqref{1-24}. Actually, we have the following optimal estimates of $u-\chi$:
\begin{cor}\label{main.cor}
Under the same assumptions in Theorem~\ref{main2}, if $\tilde{\nu_{0}}\neq0$ and $\tilde{\nu_{1}}\neq0$, then the following estimates 
\begin{align}
\label{1-26}
\|u(\cdot, t)-\chi(\cdot, t)\|_{L^{\infty}} \sim \begin{cases}
(1+t)^{-\frac{\min\{\alpha, \beta\}}{2}}, &1<\min\{\alpha, \beta\}<2, \\
(1+t)^{-1}\log(1+t), &\min\{\alpha, \beta\}=2
\end{cases}
\end{align}
hold for sufficiently large $t$.
\end{cor}

\begin{rem}
\rm{
The similar result for Theorem~\ref{main1} is obtained by Kitagawa \cite{Ki07} for the generalized Burgers equation. For Theorem~\ref{main2}, the author in \cite{F19-2} obtained the similar result for the generalized KdV-Burgers equation. 
}
\end{rem}

This paper is organized as follows. In Section 2, we introduce the global existence of solutions to \eqref{1-2} and prepare a couple of lemmas for an auxiliary problem. In Section 3, we drive the upper bound estimates of $u-\chi$, i.e. we prove Theorem~\ref{main1}. Finally, we give the proof of our second main result Theorem~\ref{main2}. The proof of Theorem~\ref{main2} is divided into two parts. We extract the second asymptotic profiles $Z(x,t)$ and $Z(x,t)+V(x,t)$ according to the decaying rate of the initial data in Section 4 and Section 5. It is the main novelty of this paper. 

\vskip10pt
\par\noindent
\textbf{\bf{Notations.}} In this paper, for $1\le p \le \infty$, $L^{p}(\R)$ denotes the usual Lebesgue spaces. In the following, for $f, g \in L^{2}(\R)\cap L^{1}(\R)$, we denote the Fourier transform of $f$ and the inverse Fourier transform of $g$ as follows:
\begin{align*}
\hat{f}(\xi)\equiv \mathcal{F}[f](\xi)=\frac{1}{\sqrt{2\pi}}\int_{\R}e^{-ix\xi}f(x)dx, \ \ \check{g}(x)\equiv \mathcal{F}^{-1}[g](x)=\frac{1}{\sqrt{2\pi}}\int_{\R}e^{ix\xi}g(\xi)d\xi.
\end{align*}
Then, for $s\ge0$, we define the Sobolev spaces by 
\begin{equation*}
H^{s}(\R)\equiv \biggl\{f\in L^{2}(\R); \ \|f\|_{H^{s}}\equiv \biggl(\int_{\R}(1+|\xi|^{2})^{s}|\hat{f}(\xi)|^{2}d\xi\biggl)^{1/2}<\infty \biggl\}. 
\end{equation*}
To express Sobolev spaces, for $1\le p\le \infty$, we also set 
\begin{equation*}
W^{m,p}(\R)\equiv \biggl\{ f\in L^{p}(\R); \ \|f\|_{W^{m,p}} \equiv \biggl(\sum_{n=0}^{m}\| \p_{x}^{n}f\|_{L^{p}}^{p}\biggl)^{1/p}<\infty \biggl\}.
\end{equation*}

Throughout this paper, $C$ denotes various positive constants, which may vary from line to line during computations. Also, it may depends on norm of the initial data or other parameters. However, we note that it does not depend on the space and time variable $x$ and $t$. Finally, for positive functions $f(t)$ and $g(t)$, we denote $f(t)\sim g(t)$ if there exist positive constants $c_{0}$ and $C_{0}$ independent of $t$ such that $c_{0}g(t)\le f(t)\le C_{0}g(t)$ holds. 

\section{Preliminaries}

In this section, we prepare a couple of lemmas to prove the main theorems. First, we shall mention the global existence and the decay estimates for the solutions to \eqref{1-2}. Now, we consider the initial value problem for the linear damped wave equation:
\begin{align}
\label{2-1}
\begin{split}
&u_{tt}-u_{xx}+u_{t}+au_{x}=0, \ x\in \R, \ t>0, \\
&u(x, 0)=u_{0}(x), \ u_{t}(x, 0)=u_{1}(x),\ x\in \R.
\end{split}
\end{align}
By taking the Fourier transform for \eqref{2-1}, it follows that 
\begin{equation*}
\hat{u}(\xi, t)=\hat{G}(\xi, t)(\hat{u}_{0}(\xi)+\hat{u}_{1}(\xi))+\p_{t}\hat{G}(\xi, t)\hat{u}_{0}(\xi), 
\end{equation*}
where
\begin{equation}
\label{2-2}
\hat{G}(\xi, t)\equiv\frac{1}{\lambda_{1}(\xi)-\lambda_{2}(\xi)}(e^{\lambda_{1}(\xi)t}-e^{\lambda_{2}(\xi)t}),
\end{equation}
\begin{equation*}
\lambda_{1}(\xi)\equiv\frac{1}{2}(-1+\sqrt{1-4(\xi^{2}+ai\xi)}), \ \ \lambda_{2}(\xi)\equiv\frac{1}{2}(-1-\sqrt{1-4(\xi^{2}+ai\xi)}).
\end{equation*}
Therefore, the solution of \eqref{2-1} can be expressed as follows:
\begin{equation*}
u(t)=G(t)*(u_{0}+u_{1})+\p_{t}G(t)*u_{0},
\end{equation*}
where we set 
\begin{equation}
\label{2-3}
G(x, t)\equiv \mathcal{F}^{-1}[\hat{G}(\cdot, t)](x).
\end{equation}
For this function $G(x, t)$, we can show the following decay estimate (for the proof, see Ueda and Kawashima \cite{UK07} and Kato and Ueda \cite{KU17}).
\begin{lem}\label{LE}
Let $1\le q\le p\le \infty$. Then the following $L^{p}$-$L^{q}$ estimates hold$:$ 
\begin{align}
\label{2-4}
\|G(t)* \phi\|_{L^{p}} &\le C(1+t)^{-\frac{1}{2}(\frac{1}{q}-\frac{1}{p})}\|\phi\|_{L^{q}}, \ \ t\ge0,\\
\label{2-5}
\|\p_{t}^{k}\p_{x}^{l}G(t)*\phi\|_{L^{p}}&\le C(1+t)^{-\frac{1}{2}(\frac{1}{q}-\frac{1}{p})-\frac{k+l}{2}}\|\phi\|_{L^{q}}+Ce^{-c_{0}t}\|\phi\|_{W^{k+l-1,p}}, \ \ t\ge0, 
\end{align} 
for $m+l\ge1$, where $G(x, t)$ and $G_{0}(x, t)$ are defined by \eqref{2-3} and \eqref{1-22}, respectively. Moreover, the solutions operator $G(t)*$ is approximated by $G_{0}(t)*$ in the following sense: 
\begin{align}
\label{2-6}
\|(G-G_{0})(t)* \phi\|_{L^{p}} &\le Ct^{-\frac{1}{2}(\frac{1}{q}-\frac{1}{p})}(1+t)^{-\frac{1}{2}}\|\phi\|_{L^{q}}, \ \ t>0,\\
\label{2-7}
\|\p_{t}^{k}\p_{x}^{l}(G-G_{0})(t)*\phi\|_{L^{p}}&\le Ct^{-\frac{1}{2}(\frac{1}{q}-\frac{1}{p})-\frac{k+l}{2}}(1+t)^{-\frac{1}{2}}\|\phi\|_{L^{q}}+Ce^{-c_{0}t}\|\phi\|_{W^{k+l-1,p}}, \ \ t>0, 
\end{align} 
for $k+l\ge1$. Here $c_{0}$ is a positive constant. 
\end{lem}
\noindent
Applying the Duhamel principle to \eqref{1-2}, we obtain  
\begin{equation}
\label{2-8}
u(t)=G(t)*(u_{0}+u_{1})+\p_{t}G(t)*u_{0}-\int_{0}^{t}G(t-\tau)*(g(u)_{x})(\tau)d\tau, 
\end{equation}
where $g(u)\equiv \frac{b}{2}u^{2}+\frac{c}{3!}u^{3}$. Therefore, by using Lemma~\ref{LE}, we obtain the global existence of the solutions to \eqref{1-2} as in the following proposition. The proof of it is given by a standard argument which is based on the contraction mapping principle (see e.g. Proposition 3.1 in \cite{KU17}):
\begin{prop}\label{GE}
Let $s$ be a positive integer and $1\le p\le \infty$. Suppose that $u_{0}\in W^{s, p}(\R)\cap L^{1}(\R)$, $u_{1}\in W^{s-1}(\R)\cap L^{1}(\R)$ and $E_{0}^{(s, p)}\equiv\|u_{0}\|_{W^{s, p}}+\|u_{0}\|_{L^{1}}+\|u_{1}\|_{W^{s-1, p}}+\|u_{1}\|_{L^{1}}$ is sufficiently small. Then $(1.2)$ has a unique global solution $u(x, t)$ with
\begin{equation*}
u\in \begin{cases}
\bigcap_{k=0}^{\sigma}C^{k}([0, \infty); W^{s-k, p})\cap C([0, \infty); L^{1}), &1\le p <\infty,\\
\bigcap_{k=0}^{\sigma}W^{k, \infty}(0, \infty; W^{s-k, \infty})\cap C([0, \infty); L^{1}), &p=\infty,\end{cases}
\end{equation*}
where $\sigma=\min\{2, s\}$. Moreover, the estimate
\begin{align}
\label{2-9}
\|u(\cdot, t)\|_{L^{q}}\le CE_{0}^{(1, p)}(1+t)^{-\frac{1}{2}+\frac{1}{2q}}
\end{align}
holds for any $q$ with $1\le q\le \infty$, and the estimate 
\begin{align}
\label{2-10}
\|\p_{t}^{k}\p_{x}^{l}u(\cdot, t)\|_{L^{p}}\le CE_{0}^{(s, p)}(1+t)^{-\frac{1}{2}+\frac{1}{2p}-\frac{k+l}{2}}
\end{align}
holds for $0\le k\le2$ and $l\ge 0$ with $0\le k+l\le s$.
\end{prop}
\noindent

Next, we treat the nonlinear diffusion wave $\chi(x, t)$ defined by \eqref{1-4}, and the heat kernel $G_{0}(x, t)$ defined by \eqref{1-22}. For $\chi(x, t)$, it is easy to see that 
\begin{equation}
\label{2-11}
|\chi(x, t)| \le C|M|(1+t)^{-\frac{1}{2}}e^{-\frac{(x-at)^{2}}{4\mu(1+t)}}, \ x\in \R, \ t\ge0.
\end{equation}
Moreover, $\chi(x, t)$ satisfies the following $L^{p}$-decay estimate (for the proof, see Lemma 4.3 in~\cite{KU17}).
\begin{lem}\label{chi.decay}
Let $k$, $l$ and $m$ be non-negative integers. Then, for $|M|\le1$ and $p\in[1, \infty]$, we have
\begin{equation}
\label{2-12}
\| \p_{t}^{k}\p_{x}^{l}(\p_{t}+a\p_{x})^{m}\chi(\cdot, t)\|_{L^{p}}\le C|M|(1+t)^{-\frac{1}{2}+\frac{1}{2p}-\frac{k+l+2m}{2}}, \ \ t\ge0.
\end{equation}
\end{lem}
\noindent
On the other hand, we have the following estimates for $G_{0}(x, t)$:
\begin{lem}\label{heat.decay}
Let $k$ and $l$ be nonnegative integers. Then, for $p\in [1, \infty]$, we have
\begin{equation}
\label{2-13}
\| \p_{t}^{k}\p_{x}^{l}G_{0}(\cdot, t)\|_{L^{p}}\le Ct^{-\frac{1}{2}+\frac{1}{2p}-\frac{k+l}{2}}, \ \ t>0.
\end{equation}
Moreover, if $\int_{\R}\phi(x)dx=0$ and 
\begin{equation}
\label{2-14}
\exists \gamma>1, \ \ \exists C>0 \ \ s.t. \ \ |\phi(x)|\le C(1+|x|)^{-\gamma}, \ \ x\in \R,
\end{equation}
then we have
\begin{align}
\label{2-15}
\|\p_{t}^{k}\p_{x}^{l}G_{0}(t)*\phi\|_{L^{p}}\le C\begin{cases}
t^{-\frac{\gamma}{2}+\frac{1}{2p}-\frac{k+l}{2}}, &t>0, \ 1<\gamma<2,\\
t^{-1+\frac{1}{2p}-\frac{k+l}{2}}\log(2+t), &t>0, \ \gamma=2.
\end{cases}
\end{align}
\end{lem}
\begin{proof}
We shall prove only \eqref{2-15} because \eqref{2-13} can be shown easily. Since $\int_{\R}\phi(x)dx=0$, we see that 
\begin{align}
\label{2-16}
\begin{split}
\p_{t}^{k}\p_{x}^{l}G_{0}(t)*\phi(x)&=\int_{\R}\p_{t}^{k}\p_{x}^{l}G_{0}(x-y, t)\phi(y)dy-\biggl(\int_{\R}\phi(y)dy\biggl)\p_{t}^{k}\p_{x}^{l}G_{0}(x, t) \\
&=\int_{\R}\p_{t}^{k}\p_{x}^{l}(G_{0}(x-y, t)-G_{0}(x, t))\phi(y)dy\\
&=\biggl(\int_{|y|\ge\sqrt{t}}+\int_{|y|\le\sqrt{t}}\biggl)\p_{t}^{k}\p_{x}^{l}(G_{0}(x-y, t)-G_{0}(x, t))\phi(y)dy\\
&\equiv I_{1}(x, t)+I_{2}(x, t).
\end{split}
\end{align}
Then, by using \eqref{2-13} and \eqref{2-14}, we have
\begin{align}
\label{2-17}
\begin{split}
\|I_{1}(\cdot, t)\|_{L^{p}}&\le \biggl(\int_{-\infty}^{-\sqrt{t}}+\int_{\sqrt{t}}^{\infty}\biggl)\|\p_{t}^{k}\p_{x}^{l}(G_{0}(\cdot-y, t)-G_{0}(\cdot, t))\|_{L^{p}}(1+|y|)^{-\gamma}dy\\
&\le Ct^{-\frac{1}{2}+\frac{1}{2p}-\frac{k+l}{2}}\biggl(\int_{-\infty}^{-\sqrt{t}}(1-y)^{-\gamma}dy+\int_{\sqrt{t}}^{\infty}(1+y)^{-\gamma}dy\biggl)\\
&=2Ct^{-\frac{1}{2}+\frac{1}{2p}-\frac{k+l}{2}}\biggl[\frac{1}{1-\gamma}(1+y)^{1-\gamma}\biggl]_{\sqrt{t}}^{\infty}\\
&\le Ct^{-\frac{1}{2}+\frac{1}{2p}-\frac{k+l}{2}-\frac{\gamma-1}{2}}=Ct^{-\frac{\gamma}{2}+\frac{1}{2p}-\frac{k+l}{2}}. 
\end{split}
\end{align}
On the other hand, from the mean value theorem, there exists $\theta \in (0, 1)$ such that 
\begin{equation*}
I_{2}(x, t)=\int_{|y|\le \sqrt{t}}(\p_{t}^{k}\p_{x}^{l+1}G_{0})(x-\theta y, t)(-y)\phi(y)dy.
\end{equation*}
Therefore, we obtain from \eqref{2-13} and \eqref{2-14}
\begin{align}
\label{2-18}
\begin{split}
\|I_{2}(\cdot, t)\|_{L^{p}}&\le \int_{|y|\le \sqrt{t}}\|(\p_{t}^{k}\p_{x}^{l+1}G_{0})(\cdot-\theta y, t)\|_{L^{p}}(1+|y|)^{-(\gamma-1)}dy\\
&\le Ct^{-1+\frac{1}{2p}-\frac{k+l}{2}}\int_{0}^{\sqrt{t}}(1+y)^{-(\gamma-1)}dy\equiv Ct^{-1+\frac{1}{2p}-\frac{k+l}{2}}J(t).
\end{split}
\end{align}
For $J(t)$, we can easily show
\begin{equation}
\label{2-19}
J(t)\begin{cases}
\le&\d \int_{0}^{\sqrt{t}}y^{-(\gamma-1)}dy=\biggl[\frac{1}{2-\gamma}y^{2-\gamma}\biggl]_{0}^{\sqrt{t}}=\frac{1}{2-\gamma}t^{1-\frac{\gamma}{2}}, \ 1<\gamma<2, \\
=&\d \int_{0}^{\sqrt{t}}(1+y)^{-1}dy=\log(1+\sqrt{t})\le \log(2+t), \ \gamma=2.
\end{cases}
\end{equation}
Thus, from \eqref{2-18} and \eqref{2-19} we have 
\begin{align}
\label{2-20}
\|I_{2}(\cdot, t)\|_{L^{p}}\le C\begin{cases}
t^{-\frac{\gamma}{2}+\frac{1}{2p}-\frac{k+l}{2}}, &t>0, \ 1<\gamma<2,\\
t^{-1+\frac{1}{2p}-\frac{k+l}{2}}\log(2+t), &t>0, \ \gamma=2.
\end{cases}
\end{align}
Combining \eqref{2-16}, \eqref{2-17} and \eqref{2-20}, we obtain \eqref{2-15}. 
\end{proof}

In the rest of this section, let us prepare the ingredients to prove Theorem~\ref{main2}. First, we consider the function $\eta(x, t)$ defined by \eqref{1-22} and an auxiliary problem. First, for the function $\eta(x, t)$, we can easily obtain that  
\begin{align}
\label{2-21}
&\min \{1, e^{\frac{bM}{2\mu}}\} \le \eta(x, t) \le \max \{1, e^{\frac{bM}{2\mu}}\}, \\ 
\label{2-22}
&\min \{1, e^{-\frac{bM}{2\mu}}\} \le \eta(x, t)^{-1} \le \max \{1, e^{-\frac{bM}{2\mu}}\}.
\end{align}
Moreover, by using Lemma~\ref{chi.decay}, we have the following $L^{p}$-decay estimate (for the proof, see Corollary~2.3 in~\cite{Ka07} or Lemma~5.4 in~\cite{KU17}).
\begin{lem}\label{eta.decay}
Let $l$ be a positive integer and $p\in[1, \infty]$. If $|M| \le1$, then we have 
\begin{align}
\label{2-23}
\| \p^{l}_{x}\eta(\cdot, t)\|_{L^{p}}+\| \p^{l}_{x}(\eta(\cdot, t)^{-1})\|_{L^{p}}&\le C|M|(1+t)^{-\frac{1}{2}(1-\frac{1}{p})-\frac{l}{2}+\frac{1}{2}}, \ \ t\ge0.
\end{align}
\end{lem}
\noindent
In the proof of Theorem~\ref{main2}, we examine the second asymptotic profile of the solution to \eqref{1-2}. To analyze the second asymptotic profile, we set $\psi \equiv u+u_{t}-\chi$. Recalling $\mu=1-a^{2}$, the perturbation $\psi(x, t)$ satisfies the following equation: 
\begin{align*}
\begin{split}
&\psi_{t}+a\psi_{x}+(b\chi \psi)_{x}-\mu \psi_{xx}\\
&=a\p_{x}(\p_{t}+a\p_{x})(u-\chi)+a\p_{x}(\p_{t}+a\p_{x})\chi-\frac{b}{2}\p_{x}((u-\chi)^{2})-\frac{c}{3!}\p_{x}(u^{3})+\p_{x}(b\chi u_{t}-\mu u_{tx}).
\end{split}
\end{align*}
To analyze the above equation, we prepare the following auxiliary problem:
\begin{align}
\label{2-24}
\begin{split}
z_{t}+az_{x}+(b\chi z)_{x}-\mu z_{xx}&=\p_{x}\lambda(x, t), \ x\in \R, \ t>0, \\
z(x, 0)&=z_{0}(x) , \ \ x\in \R, 
\end{split}
\end{align}
where $\lambda(x, t)$ is a given regular function decaying at spatial infinity. If we take the new valuable $\tilde{x}\equiv x-at$, and set $\tilde{z}(\tilde{x}, t)\equiv z(x, t)$, $\tilde{\chi}(\tilde{x}, t)\equiv \chi(x, t)$, $\tilde{\lambda}(\tilde{x}, t)\equiv \lambda(x, t)$ and $\tilde{z}_{0}(\tilde{x})\equiv z_{0}(\tilde{x})$, then \eqref{2-24} can be rewritten as follows: 
\begin{align}
\label{2-25}
\begin{split}
\tilde{z}_{t}+(b\tilde{\chi} \tilde{z})_{\tilde{x}}-\mu \tilde{z}_{\tilde{x}\tilde{x}}&=\p_{\tilde{x}}\tilde{\lambda}(\tilde{x}, t), \ \tilde{x}\in \R, \ t>0, \\
\tilde{z}(\tilde{x}, 0)&=\tilde{z}_{0}(\tilde{x}) , \ \ \tilde{x}\in \R. 
\end{split}
\end{align}
Therefore, if we set 
\begin{align}
\label{2-26}
\begin{split}
U[h](x, t, \tau)\equiv\int_{\R}\p_{x}(G_{0}(x-y, t-\tau)\eta(x, t))(\eta(y, \tau))^{-1}\biggl(\int_{-\infty}^{y}h(\xi)d\xi\biggl)dy&,\\
x\in \R, \ 0\le \tau<t,& 
\end{split}
\end{align}
then, applying Lemma 3.3 in~\cite{F19-1} or Lemma 3.1 in~\cite{Ka07} to \eqref{2-25}, we can deduce the following representation formula for \eqref{2-24}: 
\begin{lem}\label{RF}
Let $z_{0}(x)$ be a sufficiently regular function decaying at spatial infinity. Then we can get the smooth solution of  \eqref{2-24} which satisfies the following formula:  
\begin{equation}
\label{2-27}
z(x, t)=U[z_{0}](x, t, 0)+\int_{0}^{t}U[\p_{x}\lambda(\tau)](x, t, \tau)d\tau, \ x\in \R, \ t>0.
\end{equation}
\end{lem}
\noindent
This explicit representation formula \eqref{2-27} plays an important roles in the proof of Theorem~\ref{main2}. Especially, the following estimate can be obtained:
\begin{lem}\label{N.decay}
Assume that $|M| \le1$. Let $1\le p, q \le \infty$ and $\frac{1}{p}+\frac{1}{q}=1$. Then the following estimate 
\begin{align}
\label{2-28}
\begin{split}
\| U[\p_{x} \lambda(\tau)](\cdot, t, \tau)\|_{L^{\infty}} \le C\sum_{n=0}^{1}(1+t)^{-\frac{1}{2}+\frac{n}{2}}(t-\tau)^{-\frac{1}{2}+\frac{1}{2p}-\frac{n}{2}}\| \lambda(\cdot, \tau)\|_{L^{q}}
\end{split}
\end{align}
holds for $t>\tau$. 
\end{lem}
\begin{proof}
From \eqref{2-26}, we obtain
\begin{align*}
\begin{split}
U[\p_{x} \lambda(\tau)](x, t, \tau)&=\int_{\R}\p_{x}(G_{0}(x-y, t-\tau)\eta(x, t))(\eta(y, \tau))^{-1}\lambda(y, \tau)dy \\
&=\sum_{n=0}^{1}\p_{x}^{1-n}\eta(x, t)\int_{\R}\p_{x}^{n}G_{0}(x-y, t-\tau)(\eta(y, \tau))^{-1}\lambda(y, \tau)dy.
\end{split}
\end{align*}
By using Young's inequality, Lemma 2.5, \eqref{2-21}, \eqref{2-13} and \eqref{2-22}, we obtain 
\begin{align*}
\begin{split}
\|U[\p_{x} \lambda(\tau)](\cdot, t, \tau)\|_{L^{\infty}}&\le C\sum_{n=0}^{1}\|\p_{x}^{1-n}\eta(\cdot, t)\|_{L^{\infty}}\| \p_{x}^{n}G_{0}(t-\tau)*(\eta^{-1}\lambda)(\tau)\|_{L^{\infty}}\\
&\le C\sum_{n=0}^{1}(1+t)^{-\frac{1}{2}+\frac{n}{2}}(t-\tau)^{-\frac{1}{2}+\frac{1}{2p}-\frac{n}{2}}\| \lambda(\cdot, \tau)\|_{L^{q}}.
\end{split}
\end{align*}
\end{proof}

\newpage
\section{Proof of Theorem~\ref{main1}}

The purpose of this section is to prove Theorem~\ref{main1}. In order to obtain the upper bound of $u-\chi$, we rewrite the differential equations \eqref{1-2} and \eqref{1-6} as follows:
\begin{align}
\label{3-1}
u(t)&=G(t)*(u_{0}+u_{1})+\p_{t}G(t)*u_{0}-\int_{0}^{t}G(t-\tau)*(g(u)_{x})(\tau)d\tau, \\
\label{3-2}
\chi(t)&=G_{0}(t)*\chi_{0}-\frac{b}{2}\int_{0}^{t}G_{0}(t-\tau)*((\chi^{2})_{x})(\tau)d\tau, 
\end{align}
where $g(u)=\frac{b}{2}u^{2}+\frac{c}{3!}u^{3}$ and $\chi_{0}(x)=\chi(x, 0)$. Therefore, if we set 
\begin{equation}
\label{3-3}
\phi(x, t)\equiv u(x, t)-\chi(x, t),
\end{equation}
then $\phi(x, t)$ satisfies the following relation: 
\begin{align}
\label{3-4}
\begin{split}
\phi(t)=&\ (G-G_{0})(t)*(u_{0}+u_{1})+G_{0}(t)*(u_{0}+u_{1}-\chi_{0})\\
&+\p_{t}G(t)*u_{0}-\frac{c}{3!}\int_{0}^{t}G(t-\tau)*((u^{3})_{x})(\tau)d\tau\\
&-\frac{b}{2}\int_{0}^{t}(G-G_{0})(t-\tau)*((u^{2})_{x})(\tau)d\tau-\frac{b}{2}\int_{0}^{t}G_{0}(t-\tau)*((u^{2}-\chi^{2})_{x})(\tau)d\tau\\
=&\ I_{1}+I_{2}+I_{3}+I_{4}+I_{5}+I_{6}. 
\end{split}
\end{align}

Now, let us prove Theorem~\ref{main1}. Our first step to show Theorem~\ref{main1} is to derive the following proposition:
\begin{prop}\label{prop.L1}
Assume the same conditions on $u_{0}$ and $u_{1}$ in Theorem~\ref{main1} are valid. Then, for any $\e>0$, we have 
\begin{align}
\label{3-5}
\|\phi(\cdot, t)\|_{L^{1}}\le C\begin{cases}
(1+t)^{-\frac{\min\{\alpha, \beta\}}{2}+\frac{1}{2}}, &t\ge0, \ 1<\min\{\alpha, \beta\}<2,\\
(1+t)^{-\frac{1}{2}+\e}, &t\ge0, \ \min\{\alpha, \beta\}=2,
\end{cases}
\end{align}
where $\phi(x, t)$ is defined by \eqref{3-3}.
\end{prop} 
\begin{proof}
We set 
\begin{align}
\label{3-6}
M(T)\equiv \begin{cases}\d
\sup_{0\le t\le T}(1+t)^{\frac{\gamma-1}{2}}\|\phi(\cdot, t)\|_{L^{1}}, &1<\gamma<2,\\
\d \sup_{0\le t\le T}(1+t)^{\frac{1}{2}-\e}\|\phi(\cdot, t)\|_{L^{1}}, &\gamma=2, 
\end{cases}
\end{align}
where $\gamma \equiv \min\{\alpha, \beta\}$ and $\e$ is any fixed constant such that $0<\e<\frac{1}{2}$. It suffices to estimate the each term of the right hand side of \eqref{3-4}. For the first term, from \eqref{2-6}, we have 
\begin{equation}
\label{3-7}
\|I_{1}(\cdot, t)\|_{L^{1}}\le C(1+t)^{-\frac{1}{2}}(\|u_{0}\|_{L^{1}}+\|u_{1}\|_{L^{1}}), \ t>0.
\end{equation}
Also, since $\int_{\R}(u_{0}+u_{1}-\chi_{0})dx=0$, \eqref{1-3}, \eqref{1-4} and \eqref{1-5}, by using \eqref{2-15}, it follows that 
\begin{align}
\label{3-8}
\|I_{2}(\cdot, t)\|_{L^{1}}\le C_{0}\begin{cases}
(1+t)^{-\frac{\gamma-1}{2}}, &t\ge1, \ 1<\gamma<2,\\
(1+t)^{-\frac{1}{2}}\log(2+t), &t\ge1, \ \gamma=2,
\end{cases}
\end{align}
where $C_{0}$ is a positive constant. For $I_{3}$, applying \eqref{2-5}, we obtain 
\begin{equation}
\label{3-9}
\|I_{3}(\cdot, t)\|_{L^{1}}\le C(1+t)^{-\frac{1}{2}}\|u_{0}\|_{L^{1}}, \ t\ge0.
\end{equation}
Next, we evaluate $I_{4}$. Applying \eqref{2-5} and \eqref{2-9}, we have
\begin{align}
\label{3-10}
\begin{split}
\|I_{4}(\cdot, t)\|_{L^{1}}&\le C\int_{0}^{t}\|\p_{x}G(t-\tau)*u^{3}(\tau)\|_{L^{1}}d\tau\\
&\le C\int_{0}^{t}(1+t-\tau)^{-\frac{1}{2}}\|u^{3}(\cdot, \tau)\|_{L^{1}}d\tau\\
&\le C\int_{0}^{t}(1+t-\tau)^{-\frac{1}{2}}\|u(\cdot, \tau)\|_{L^{\infty}}^{2}\|u(\cdot, \tau)\|_{L^{1}}d\tau\\
&\le CE_{0}^{(1, p)}\int_{0}^{t}(1+t-\tau)^{-\frac{1}{2}}(1+\tau)^{-1}d\tau \le CE_{0}^{(1, p)}(1+t)^{-\frac{1}{2}}\log(2+t), \ t\ge0. 
\end{split}
\end{align}
We note that $I_{4}$ does not appear if $c=0$. For $I_{5}$, by using \eqref{2-7} and \eqref{2-9}, similarly we have 
\begin{align}
\label{3-11}
\begin{split}
\|I_{5}(\cdot, t)\|_{L^{1}}&\le C\int_{0}^{t}\|\p_{x}(G-G_{0})(t-\tau)*u^{2}(\tau)\|_{L^{1}}d\tau\\
&\le C\int_{0}^{t}(t-\tau)^{-\frac{1}{2}}(1+t-\tau)^{-\frac{1}{2}}\|u^{2}(\cdot, \tau)\|_{L^{1}}d\tau\\
&\le CE_{0}^{(1, p)}\int_{0}^{t}(t-\tau)^{-\frac{1}{2}}(1+t-\tau)^{-\frac{1}{2}}(1+\tau)^{-\frac{1}{2}}d\tau\\
&\le CE_{0}^{(1, p)}(1+t)^{-\frac{1}{2}}\log(2+t), \ t\ge1. 
\end{split}
\end{align}
Finally, we evaluate $I_{6}$. From Young's inequality, \eqref{2-13}, \eqref{2-9}, \eqref{2-12}, \eqref{3-3} and \eqref{3-6}, we obtain 
\begin{align}
\label{3-12}
\begin{split}
\|I_{6}(\cdot, t)\|_{L^{1}}&\le C\int_{0}^{t}\|\p_{x}G_{0}(t-\tau)*(u^{2}-\chi^{2})(\tau)\|_{L^{1}}d\tau\\
&\le C\int_{0}^{t}(t-\tau)^{-\frac{1}{2}}\|(u^{2}-\chi^{2})(\cdot, \tau)\|_{L^{1}}d\tau\\
&\le C\int_{0}^{t}(t-\tau)^{-\frac{1}{2}}(\|u(\cdot, \tau)\|_{L^{\infty}}+\|\chi(\cdot, \tau)\|_{L^{\infty}})\|\phi(\cdot, \tau)\|_{L^{1}}d\tau\\
&\le CE_{0}^{(1, p)}M(T)\begin{cases}\d \int_{0}^{t}(t-\tau)^{-\frac{1}{2}}(1+\tau)^{-\frac{\gamma}{2}}d\tau, &1<\gamma<2,\\
\d \int_{0}^{t}(t-\tau)^{-\frac{1}{2}}(1+\tau)^{-1+\e}d\tau, &\gamma=2
\end{cases}\\
&\le CE_{0}^{(1, p)}M(T)\begin{cases}(1+t)^{-\frac{\gamma-1}{2}}, &t\ge1, \ 1<\gamma<2, \\
 (1+t)^{-\frac{1}{2}+\e}, &t\ge1, \ \gamma=2.
 \end{cases}
\end{split}
\end{align}
Therefore, combining \eqref{3-4} and \eqref{3-7} through \eqref{3-12}, we have
\begin{align}
\label{3-13}
\begin{split}
\|\phi(\cdot, t)\|_{L^{1}}\le&\ CE_{0}^{(1, p)}(1+t)^{-\frac{1}{2}}+CE_{0}^{(1, p)}(1+t)^{-\frac{1}{2}}\log(2+t)\\
&+C_{0}\begin{cases}(1+t)^{-\frac{\gamma-1}{2}}, &1\le t\le T, \ 1<\gamma<2, \\
 (1+t)^{-\frac{1}{2}}\log(2+t), &1\le t\le T, \ \gamma=2
 \end{cases}\\
&+CE_{0}^{(1, p)}M(T)\begin{cases}(1+t)^{-\frac{\gamma-1}{2}}, &1\le t\le T, \ 1<\gamma<2, \\
 (1+t)^{-\frac{1}{2}+\e}, &1\le t\le T, \ \gamma=2.
 \end{cases}
\end{split}
\end{align}
For $0\le t\le 1$, from \eqref{2-9}, \eqref{2-12} and $|M|\le E_{0}^{(1, p)}$, we obtain 
\begin{equation}
\label{3-14}
\|\phi(\cdot, t)\|_{L^{1}}\le \|u(\cdot, t)\|_{L^{1}}+\|\chi(\cdot, t)\|_{L^{1}}\le CE_{0}^{(1, p)}, \ 0\le t\le 1.
\end{equation}
Since $\log(2+t)\le C(1+t)^{\e}$, combining \eqref{3-13} and \eqref{3-14}, we arrive at 
\begin{equation*}
M(T)\le CE_{0}^{(1, p)}+C_{0}+C_{1}E_{0}^{(1, p)}M(T),
\end{equation*}
where $C_{1}$ is a positive constant. Therefore, we obtain the desired estimate 
\begin{equation*}
M(T)\le 2CE_{0}^{(1, p)}+2C_{0}
\end{equation*}
if $E_{0}^{(1, p)}$ is small that $C_{1}E_{0}^{(1, p)}\le\frac{1}{2}$. This completes the proof. 
\end{proof}

Next, we shall derive $L^{\infty}$-estimate of $u-\chi$. Actually, we have the following proposition:
\begin{prop}\label{prop.Linf}
Assume the same conditions on $u_{0}$ and $u_{1}$ in Theorem~\ref{main1} are valid. Then, for any $\e>0$, we have 
\begin{align}
\label{3-15}
\|\phi(\cdot, t)\|_{L^{\infty}}\le C\begin{cases}
(1+t)^{-\frac{\min\{\alpha, \beta\}}{2}}, &t\ge0, \ 1<\min\{\alpha, \beta\}<2,\\
(1+t)^{-1+\e}, &t\ge0, \ \min\{\alpha, \beta\}=2,
\end{cases}
\end{align}
where $\phi(x, t)$ is defined by \eqref{3-3}.
\end{prop} 
\begin{proof}
We set 
\begin{align}
\label{3-16}
N(T)\equiv \begin{cases}\d
\sup_{0\le t\le T}(1+t)^{\frac{\gamma}{2}}\|\phi(\cdot, t)\|_{L^{\infty}}, &1<\gamma<2,\\
\d \sup_{0\le t\le T}(1+t)^{1-\e}\|\phi(\cdot, t)\|_{L^{\infty}}, &\gamma=2, 
\end{cases}
\end{align}
where $\gamma \equiv \min\{\alpha, \beta\}$ and $\e$ is any fixed constant such that $0<\e<1$. We evaluate the each term of the right hand side of \eqref{3-4}. For $I_{1}$, from \eqref{2-6}, we have 
\begin{equation}
\label{3-17}
\|I_{1}(\cdot, t)\|_{L^{\infty}}\le Ct^{-\frac{1}{2}}(1+t)^{-\frac{1}{2}}(\|u_{0}\|_{L^{1}}+\|u_{1}\|_{L^{1}})\le CE_{0}^{(1, p)}(1+t)^{-1}, \ t\ge1.
\end{equation}
Also, in the same way to get \eqref{3-8}, applying \eqref{2-15}, we see that 
\begin{align}
\label{3-18}
\|I_{2}(\cdot, t)\|_{L^{\infty}}\le C_{0}\begin{cases}
(1+t)^{-\frac{\gamma}{2}}, &t\ge1, \ 1<\gamma<2,\\
(1+t)^{-1}\log(2+t), &t\ge1, \ \gamma=2, 
\end{cases}
\end{align}
where $C_{0}$ is a positive constant. For $I_{3}$, from \eqref{2-5}, we get
\begin{equation}
\label{3-19}
\|I_{3}(\cdot, t)\|_{L^{\infty}}\le C(1+t)^{-1}\|u_{0}\|_{L^{1}}+Ce^{-c_{0}t}\|u_{0}\|_{L^{\infty}}\le CE_{0}^{(1, p)}(1+t)^{-1}, \ t\ge0.
\end{equation}
Here, we have used the Gagliardo-Nirenberg inequality $\|u_{0}\|_{L^{\infty}}\le C\|u_{0}\|_{L^{p}}^{1-\frac{1}{p}}\|u'_{0}\|_{L^{p}}^{\frac{1}{p}}\le CE_{0}^{(1, p)}$. Next, applying \eqref{2-5} and \eqref{2-9}, we have
\begin{align}
\label{3-20}
\begin{split}
\|I_{4}(\cdot, t)\|_{L^{\infty}}&\le C\int_{0}^{t}\|\p_{x}G(t-\tau)*u^{3}(\tau)\|_{L^{\infty}}d\tau\\
&\le C\int_{0}^{t}(1+t-\tau)^{-1}\|u^{3}(\cdot, \tau)\|_{L^{1}}d\tau+C\int_{0}^{t}e^{-c_{0}(t-\tau)}\|u^{3}(\cdot, \tau)\|_{L^{\infty}}d\tau\\
&\le C\int_{0}^{t}(1+t-\tau)^{-1}\|u(\cdot, \tau)\|_{L^{\infty}}^{2}\|u(\cdot, \tau)\|_{L^{1}}d\tau+C\int_{0}^{t}e^{-c_{0}(t-\tau)}\|u^{3}(\cdot, \tau)\|_{L^{\infty}}d\tau\\
&\le CE_{0}^{(1, p)}\int_{0}^{t}(1+t-\tau)^{-1}(1+\tau)^{-1}d\tau+CE_{0}^{(1, p)}\int_{0}^{t}e^{-c_{0}(t-\tau)}(1+\tau)^{-\frac{3}{2}}d\tau\\
&\le CE_{0}^{(1, p)}(1+t)^{-1}\log(2+t), \ t\ge0. 
\end{split}
\end{align}
For $I_{5}$, by using \eqref{2-6}, H$\ddot{\text{o}}$lder's inequality, \eqref{2-9} and \eqref{2-10}, we have 
\begin{align}
\label{3-21}
\begin{split}
\|I_{5}(\cdot, t)\|_{L^{\infty}}&\le C\int_{0}^{t}\|(G-G_{0})(t-\tau)*\p_{x}(u^{2})(\tau)\|_{L^{1}}d\tau\\
&\le C\int_{0}^{t}(t-\tau)^{-\frac{1}{2}}(1+t-\tau)^{-\frac{1}{2}}\|\p_{x}(u^{2}(\cdot, \tau))\|_{L^{1}}d\tau\\
&\le C\int_{0}^{t}(t-\tau)^{-\frac{1}{2}}(1+t-\tau)^{-\frac{1}{2}}\|u(\cdot, \tau)\|_{L^{q}}\|\p_{x}u(\cdot, \tau)\|_{L^{p}}d\tau \ \ \biggl(\frac{1}{p}+\frac{1}{q}=1\biggl)\\
&\le CE_{0}^{(1, p)}\int_{0}^{t}(t-\tau)^{-\frac{1}{2}}(1+t-\tau)^{-\frac{1}{2}}(1+\tau)^{-1}d\tau\\
&\le CE_{0}^{(1, p)}(1+t)^{-1}\log(2+t), \ t\ge1. 
\end{split}
\end{align}
Finally, we evaluate $I_{6}$. From Young's inequality, \eqref{2-13}, \eqref{2-9}, \eqref{2-12}, \eqref{3-3}, \eqref{3-5} and \eqref{3-16}, we obtain 
\begin{align}
\label{3-22}
\begin{split}
\|I_{6}(\cdot, t)\|_{L^{\infty}}&\le C\int_{0}^{t}\|\p_{x}G_{0}(t-\tau)*(u^{2}-\chi^{2})(\tau)\|_{L^{\infty}}d\tau\\
&\le C\int_{0}^{t/2}(t-\tau)^{-1}\|(u^{2}-\chi^{2})(\cdot, \tau)\|_{L^{1}}d\tau+C\int_{t/2}^{t}(t-\tau)^{-\frac{1}{2}}\|(u^{2}-\chi^{2})(\cdot, \tau)\|_{L^{\infty}}d\tau\\
&\le C\int_{0}^{t/2}(t-\tau)^{-1}(\|u(\cdot, \tau)\|_{L^{\infty}}\|+\chi(\cdot, \tau)\|_{L^{\infty}})\|\phi(\cdot, \tau)\|_{L^{1}}d\tau\\
&\ \ \ +C\int_{t/2}^{t}(t-\tau)^{-\frac{1}{2}}(\|u(\cdot, \tau)\|_{L^{\infty}}+\|\chi(\cdot, \tau)\|_{L^{\infty}})\|\phi(\cdot, \tau)\|_{L^{\infty}}d\tau\\
&\le CE_{0}^{(1, p)}\begin{cases}\d \int_{0}^{t/2}(t-\tau)^{-1}(1+\tau)^{-\frac{\gamma}{2}}d\tau, &1<\gamma<2,\\
\d \int_{0}^{t/2}(t-\tau)^{-1}(1+\tau)^{-1+\e}d\tau, &\gamma=2
\end{cases}\\
&\ \ \ +CE_{0}^{(1, p)}N(T)\begin{cases}\d \int_{t/2}^{t}(t-\tau)^{-\frac{1}{2}}(1+\tau)^{-\frac{\gamma+1}{2}}d\tau, &1<\gamma<2,\\
\d \int_{t/2}^{t}(t-\tau)^{-\frac{1}{2}}(1+\tau)^{-\frac{3}{2}+\e}d\tau, &\gamma=2
\end{cases}\\
&\le CE_{0}^{(1, p)}(1+N(T))\begin{cases}(1+t)^{-\frac{\gamma}{2}}, &t\ge1, \ 1<\gamma<2, \\
 (1+t)^{-1+\e}, &t\ge1, \ \gamma=2.
 \end{cases}
\end{split}
\end{align}
Therefore, combining \eqref{3-4} and \eqref{3-17} through \eqref{3-22}, we have
\begin{align}
\label{3-23}
\begin{split}
\|\phi(\cdot, t)\|_{L^{\infty}}\le&\ CE_{0}^{(1, p)}(1+t)^{-1}+CE_{0}^{(1, p)}(1+t)^{-1}\log(2+t)\\
&+C_{0}\begin{cases}(1+t)^{-\frac{\gamma}{2}}, &1\le t\le T, \ 1<\gamma<2, \\
 (1+t)^{-1}\log(2+t), &1\le t\le T, \ \gamma=2
 \end{cases}\\
&+CE_{0}^{(1, p)}(1+N(T))\begin{cases}(1+t)^{-\frac{\gamma}{2}}, &1\le t\le T, \ 1<\gamma<2, \\
 (1+t)^{-1+\e}, &1\le t\le T, \ \gamma=2.
 \end{cases}
\end{split}
\end{align}
For $0\le t\le 1$, in the same way to get \eqref{3-14}, we obtain 
\begin{equation}
\label{3-24}
\|\phi(\cdot, t)\|_{L^{\infty}}\le \|u(\cdot, t)\|_{L^{\infty}}+\|\chi(\cdot, t)\|_{L^{\infty}}\le CE_{0}^{(1, p)}, \ 0\le t\le 1.
\end{equation}
Since $\log(2+t)\le C(1+t)^{\e}$, combining \eqref{3-23} and \eqref{3-24}, it follows that  
\begin{equation*}
N(T)\le CE_{0}^{(1, p)}+C_{0}+C_{1}E_{0}^{(1, p)}N(T),
\end{equation*}
where $C_{1}$ is a positive constant. Therefore, we obtain the desired estimate 
\begin{equation*}
N(T)\le 2CE_{0}^{(1, p)}+2C_{0}
\end{equation*}
if $E_{0}^{(1, p)}$ is so small that $C_{1}E_{0}^{(1, p)}\le\frac{1}{2}$. This completes the proof. 
\end{proof}
\noindent
By virtue of Proposition~\ref{prop.L1}, Proposition~\ref{prop.Linf} and the the interpolation inequality\begin{equation*}
\|\phi(\cdot, t)\|_{L^{q}}\le \|\phi(\cdot, t)\|_{L^{\infty}}^{1-1/q}\|\phi(\cdot, t)\|_{L^{1}}^{1/q}, \ \ 1\le q\le \infty, 
\end{equation*}
We have the following corollary:
\begin{cor}\label{cor.Lq}
Assume the same conditions on $u_{0}$ and $u_{1}$ in Theorem~\ref{main1} are valid. Then, for any $\e>0$, we have 
\begin{align}
\label{3-26}
\|\phi(\cdot, t)\|_{L^{q}}\le C\begin{cases}
(1+t)^{-\frac{\min\{\alpha, \beta\}}{2}+\frac{1}{2q}}, &t\ge0, \ 1<\min\{\alpha, \beta\}<2,\\
(1+t)^{-1+\frac{1}{2q}+\e}, &t\ge0, \ \min\{\alpha, \beta\}=2
\end{cases}
\end{align}
for any $q$ with $1\le q\le \infty$, where $\phi(x, t)$ is defined by \eqref{3-3}.
\end{cor} 

Next, we derive the following $L^{p}$-decay estimate of the spatial derivatives of $u-\chi$:
\begin{prop}\label{prop.Lp}
Assume the same conditions on $u_{0}$ and $u_{1}$ in Theorem~\ref{main1} are valid. Then, for any $\e>0$, we have 
\begin{align}
\label{3-27}
\|\p_{x}^{l}\phi(\cdot, t)\|_{L^{p}}\le C\begin{cases}
(1+t)^{-\frac{\min\{\alpha, \beta\}}{2}+\frac{1}{2p}-\frac{l}{2}}, &t\ge0, \ 1<\min\{\alpha, \beta\}<2,\\
(1+t)^{-1+\frac{1}{2p}-\frac{l}{2}+\e}, &t\ge0, \ \min\{\alpha, \beta\}=2
\end{cases}
\end{align}
for $0\le l\le s$, where $\phi(x, t)$ is defined by \eqref{3-3}.
\end{prop} 
\begin{proof}
We have already shown \eqref{3-27} for $l=0$. In the following, let us prove \eqref{3-27} for $1\le l\le s$. We set 
\begin{align}
\label{3-28}
L(T)\equiv \begin{cases}\d
\sup_{0\le t\le T}\sum_{l=1}^{s}(1+t)^{\frac{\gamma}{2}-\frac{1}{2p}+\frac{l}{2}}\|\p_{x}^{l}\phi(\cdot, t)\|_{L^{p}}, &1<\gamma<2,\\
\d \sup_{0\le t\le T}\sum_{l=1}^{s}(1+t)^{1-\frac{1}{2p}+\frac{l}{2}-\e}\|\p_{x}^{l}\phi(\cdot, t)\|_{L^{p}}, &\gamma=2, 
\end{cases}
\end{align}
where $\gamma \equiv \min\{\alpha, \beta\}$ and $\e$ is any fixed constant such that $0<\e<\frac{1}{2}$. We evaluate the each term of the right hand side of \eqref{3-4}. For $I_{1}$, from \eqref{2-7}, we have 
\begin{align}
\label{3-29}
\begin{split}
\|\p_{x}^{l}I_{1}(\cdot, t)\|_{L^{p}}&\le Ct^{-\frac{1}{2}(1-\frac{1}{p})-\frac{l}{2}}(1+t)^{-\frac{1}{2}}\|u_{0}+u_{1}\|_{L^{1}}+Ce^{-c_{0}t}\|u_{0}+u_{1}\|_{W^{l-1, p}}\\
&\le CE_{0}^{(s, p)}(1+t)^{-1+\frac{1}{2p}-\frac{l}{2}}, \ t\ge1.
\end{split}
\end{align}
Also, applying \eqref{2-15}, it follows that 
\begin{align}
\label{3-30}
\|\p_{x}^{l}I_{2}(\cdot, t)\|_{L^{p}}\le C_{0}\begin{cases}
(1+t)^{-\frac{\gamma}{2}+\frac{1}{2p}-\frac{l}{2}}, &t\ge1, \ 1<\gamma<2,\\
(1+t)^{-1+\frac{1}{2p}-\frac{l}{2}}\log(2+t), &t\ge1, \ \gamma=2, 
\end{cases}
\end{align}
where $C_{0}$ is a positive constant. For $I_{3}$, from \eqref{2-5}, we get
\begin{align}
\label{3-31}
\begin{split}
\|\p_{x}^{l}I_{3}(\cdot, t)\|_{L^{p}}&\le C(1+t)^{-1+\frac{1}{2p}-\frac{l}{2}}\|u_{0}\|_{L^{1}}+Ce^{-c_{0}t}\|u_{0}\|_{W^{l, p}}\\
&\le CE_{0}^{(s, p)}(1+t)^{-1+\frac{1}{2p}-\frac{l}{2}}, \ t\ge0.
\end{split}
\end{align}
Next, applying \eqref{2-5}, \eqref{2-9} and \eqref{2-10}, we have
\begin{align}
\label{3-32}
\begin{split}
&\|\p_{x}^{l}I_{4}(\cdot, t)\|_{L^{p}}\\
&\le C\int_{0}^{t/2}\|\p_{x}^{l+1}G(t-\tau)*u^{3}(\tau)\|_{L^{p}}d\tau+C\int_{t/2}^{t}\|\p_{x}G(t-\tau)*\p_{x}^{l}(u^{3})(\tau)\|_{L^{p}}d\tau\\
&\le C\int_{0}^{t/2}(1+t-\tau)^{-1+\frac{1}{2p}-\frac{l}{2}}\|u^{3}(\cdot, \tau)\|_{L^{1}}d\tau+C\int_{0}^{t/2}e^{-c_{0}(t-\tau)}\|u^{3}(\cdot, \tau)\|_{W^{l, p}}d\tau\\
&\ \ \ +C\int_{t/2}^{t}(1+t-\tau)^{-\frac{1}{2}}\|\p_{x}^{l}(u^{3}(\cdot, \tau))\|_{L^{p}}d\tau\\
&\le CE_{0}^{(s, p)}\int_{0}^{t/2}(1+t-\tau)^{-1+\frac{1}{2p}-\frac{l}{2}}(1+\tau)^{-1}d\tau+CE_{0}^{(s, p)}\int_{0}^{t/2}e^{-c_{0}(t-\tau)}(1+\tau)^{-\frac{3}{2}+\frac{1}{2p}}d\tau\\
&\ \ \ +CE_{0}^{(s, p)}\int_{t/2}^{t}(1+t-\tau)^{-\frac{1}{2}}(1+\tau)^{-\frac{3}{2}+\frac{1}{2p}-\frac{l}{2}}d\tau\\
&\le CE_{0}^{(s, p)}(1+t)^{-1+\frac{1}{2p}-\frac{l}{2}}\log(2+t), \ t\ge0,
\end{split}
\end{align}
where we have used the estimate $\|\p_{x}^{l}(u^{3}(\cdot, t))\|_{L^{p}}\le CE_{0}^{(s, p)}(1+t)^{-\frac{3}{2}+\frac{1}{2p}-\frac{l}{2}}$. For $I_{5}$, by using \eqref{2-7}, \eqref{2-9} and \eqref{2-10}, we have 
\begin{align}
\label{3-33}
\begin{split}
\|\p_{x}^{l}I_{5}(\cdot, t)\|_{L^{p}}&\le C\int_{0}^{t/2}\|\p_{x}^{l+1}(G-G_{0})(t-\tau)*u^{2}(\tau)\|_{L^{p}}d\tau\\
&\ \ \ +C\int_{t/2}^{t}\|\p_{x}(G-G_{0})(t-\tau)*\p_{x}^{l}(u^{2})(\tau)\|_{L^{p}}d\tau\\
&\le C\int_{0}^{t/2}(t-\tau)^{-1+\frac{1}{2p}-\frac{l}{2}}(1+t-\tau)^{-\frac{1}{2}}\|u^{2}(\cdot, \tau)\|_{L^{1}}d\tau\\
&\ \ \ +C\int_{0}^{t/2}e^{-c_{0}(t-\tau)}\|u^{2}(\cdot, \tau)\|_{W^{l, p}}d\tau\\
&\ \ \ +C\int_{t/2}^{t}(t-\tau)^{-\frac{1}{2}}(1+t-\tau)^{-\frac{1}{2}}\|\p_{x}^{l}(u^{2}(\cdot, \tau))\|_{L^{p}}d\tau\\
&\le CE_{0}^{(s, p)}\int_{0}^{t/2}(t-\tau)^{-1+\frac{1}{2p}-\frac{l}{2}}(1+t-\tau)^{-\frac{1}{2}}(1+\tau)^{-\frac{1}{2}}d\tau\\
&\ \ \ +CE_{0}^{(s, p)}\int_{0}^{t/2}e^{-c_{0}(t-\tau)}(1+\tau)^{-1+\frac{1}{2p}}d\tau\\
&\ \ \ +CE_{0}^{(s, p)}\int_{t/2}^{t}(t-\tau)^{-\frac{1}{2}}(1+t-\tau)^{-\frac{1}{2}}(1+\tau)^{-1+\frac{1}{2p}-\frac{l}{2}}d\tau\\
&\le CE_{0}^{(s, p)}(1+t)^{-1+\frac{1}{2p}-\frac{l}{2}}\log(2+t), \ t\ge1,
\end{split}
\end{align}
where we have used the estimate $\|\p_{x}^{l}(u^{2}(\cdot, t))\|_{L^{p}}\le CE_{0}^{(s, p)}(1+t)^{-1+\frac{1}{2p}-\frac{l}{2}}$. Finally, we evaluate $I_{6}$. We prepare the following estimate. 
\begin{align}
\label{3-34}
\begin{split}
\|\p_{x}^{l}((u^{2}-\chi^{2})(\cdot, t))\|_{L^{p}}&=\|\p_{x}^{l}(((u+\chi)(u-\chi))(\cdot, t))\|_{L^{p}}\\
&\le C(\|\p_{x}^{l}u(\cdot, t)\|_{L^{p}}+\|\p_{x}^{l}\chi(\cdot, t)\|_{L^{p}})\|\phi(\cdot, t)\|_{L^{\infty}} \\
&\ \ \ +C\sum_{m=0}^{l-1}(\|\p_{x}^{m}u(\cdot, t)\|_{L^{\infty}}+\|\p_{x}^{m}\chi(\cdot, t)\|_{L^{\infty}})\|\p_{x}^{l-m}\phi(\cdot, t)\|_{L^{p}} \\
&\le CE_{0}^{(s, p)}(1+t)^{-\frac{1}{2}+\frac{1}{2p}-\frac{l}{2}}\begin{cases}
(1+t)^{-\frac{\gamma}{2}}, &1<\gamma<2, \\
 (1+t)^{-1+\e}, &\gamma=2
\end{cases}\\
&\ \ \ +CE_{0}^{(s, p)}L(T)\sum_{m=0}^{l-1}(1+t)^{-\frac{1}{2}-\frac{m}{2}}\begin{cases}
(1+t)^{-\frac{\gamma}{2}+\frac{1}{2p}-\frac{l-m}{2}}, &1<\gamma<2, \\
 (1+t)^{-1+\frac{1}{2p}-\frac{l-m}{2}+\e}, &\gamma=2
 \end{cases}\\
 &\le CE_{0}^{(s, p)}(1+L(T))\begin{cases}
(1+t)^{-\frac{1}{2}+\frac{1}{2p}-\frac{l}{2}-\frac{\gamma}{2}}, &1<\gamma<2, \\
 (1+t)^{-\frac{3}{2}+\frac{1}{2p}-\frac{l}{2}+\e}, &\gamma=2.
\end{cases}
\end{split}
\end{align}
Here, we have used \eqref{2-10}, \eqref{2-12}, \eqref{3-3}, \eqref{3-26}, \eqref{3-28} and the Gagliardo-Nirenberg inequality
\begin{equation*}
\|\p_{x}^{l}u(\cdot, t)\|_{L^{\infty}}\le C\|\p_{x}^{l}u(\cdot, t)\|_{L^{p}}^{1-\frac{1}{p}}\|\p_{x}^{l+1}u(\cdot, t)\|_{L^{p}}^{\frac{1}{p}}, \ \ 0\le l \le s-1.
\end{equation*}
Therefore, from Young's inequality, \eqref{2-13} and \eqref{3-34}, we obtain 
\begin{align}
\label{3-35}
\begin{split}
&\|\p_{x}^{l}I_{6}(\cdot, t)\|_{L^{p}}\\
&\le C\int_{0}^{t/2}\|\p_{x}^{l+1}G_{0}(t-\tau)*(u^{2}-\chi^{2})(\tau)\|_{L^{p}}d\tau+C\int_{t/2}^{t}\|\p_{x}G_{0}(t-\tau)*\p_{x}^{l}(u^{2}-\chi^{2})(\tau)\|_{L^{p}}d\tau\\
&\le C\int_{0}^{t/2}(t-\tau)^{-1+\frac{1}{2p}-\frac{l}{2}}\|(u^{2}-\chi^{2})(\cdot, \tau)\|_{L^{1}}d\tau+C\int_{t/2}^{t}(t-\tau)^{-\frac{1}{2}}\|\p_{x}^{l}((u^{2}-\chi^{2})(\cdot, \tau))\|_{L^{p}}d\tau\\
&\le CE_{0}^{(s, p)}(1+L(T))\begin{cases}\d \int_{0}^{t/2}(t-\tau)^{-1+\frac{1}{2p}-\frac{l}{2}}(1+\tau)^{-\frac{\gamma}{2}}d\tau, &1<\gamma<2,\\
\d \int_{0}^{t/2}(t-\tau)^{-1+\frac{1}{2p}-\frac{l}{2}}(1+\tau)^{-1+\e}d\tau, &\gamma=2
\end{cases}\\
&\ \ \ +CE_{0}^{(s, p)}(1+L(T))\begin{cases}\d \int_{t/2}^{t}(t-\tau)^{-\frac{1}{2}}(1+\tau)^{-\frac{1}{2}+\frac{1}{2p}-\frac{l}{2}-\frac{\gamma}{2}}d\tau, &1<\gamma<2,\\
\d \int_{t/2}^{t}(t-\tau)^{-\frac{1}{2}}(1+\tau)^{-\frac{3}{2}+\frac{1}{2p}-\frac{l}{2}+\e}d\tau, &\gamma=2
\end{cases}\\
&\le CE_{0}^{(s, p)}(1+L(T))\begin{cases}(1+t)^{-\frac{\gamma}{2}+\frac{1}{2p}-\frac{l}{2}}, &t\ge1, \ 1<\gamma<2, \\
 (1+t)^{-1+\frac{1}{2p}-\frac{l}{2}+\e}, &t\ge1, \ \gamma=2.
 \end{cases}
\end{split}
\end{align}
Thus, combining \eqref{3-4}, \eqref{3-29} through \eqref{3-33} and \eqref{3-35}, we obtain 
\begin{align}
\label{3-36}
\begin{split}
\|\p_{x}^{l}\phi(\cdot, t)\|_{L^{p}}\le&\ CE_{0}^{(s, p)}(1+t)^{-1+\frac{1}{2p}-\frac{l}{2}}+CE_{0}^{(s, p)}(1+t)^{-1+\frac{1}{2p}-\frac{l}{2}}\log(2+t)\\
&+C_{0}\begin{cases}
(1+t)^{-\frac{\gamma}{2}+\frac{1}{2p}-\frac{l}{2}}, &1\le t\le T, \ 1<\gamma<2,\\
(1+t)^{-1+\frac{1}{2p}-\frac{l}{2}}\log(2+t), &1\le t\le T, \ \gamma=2
\end{cases}
\\
&+CE_{0}^{(s, p)}(1+L(T))\begin{cases}(1+t)^{-\frac{\gamma}{2}+\frac{1}{2p}-\frac{l}{2}}, &1\le t\le T, \ 1<\gamma<2, \\
 (1+t)^{-1+\frac{1}{2p}-\frac{l}{2}+\e}, &1\le t\le T, \ \gamma=2.
 \end{cases}
\end{split}
\end{align}
For $0\le t\le 1$, in the same way to get \eqref{3-14} and \eqref{3-24}, we easily see  
\begin{equation}
\label{3-37}
\|\p_{x}^{l}\phi(\cdot, t)\|_{L^{p}}\le \|\p_{x}^{l}u(\cdot, t)\|_{L^{p}}+\|\p_{x}^{l}\chi(\cdot, t)\|_{L^{p}}\le CE_{0}^{(s, p)}, \ 0\le t\le 1.
\end{equation}
Summing up \eqref{3-36} and \eqref{3-37}, it follows that  
\begin{equation*}
L(T)\le CE_{0}^{(s, p)}+C_{0}+C_{1}E_{0}^{(s, p)}L(T),
\end{equation*}
where $C_{1}$ is a positive constant. Therefore, we arrive at the desired estimate 
\begin{equation*}
L(T)\le 2CE_{0}^{(s, p)}+2C_{0}
\end{equation*}
if $E_{0}^{(s, p)}$ is so small that $C_{1}E_{0}^{(s, p)}\le\frac{1}{2}$. This completes the proof. 
\end{proof}

\newpage
\begin{proof}[\rm{\bf{End of the proof of Theorem 1.1}}]
Since we have already shown \eqref{1-16} and \eqref{1-17} with $k=0, 1$ (Corollary~\ref{cor.Lq} and Proposition~\ref{prop.Lp}), we only need to prove \eqref{1-17} with $k=1, 2$. First, differentiating \eqref{3-1} with respect to $t$, then we have 
\begin{align*}
\begin{split}
\p_{t}u(t)=\p_{t}G(t)*(u_{0}+u_{1})+\p_{t}^{2}G(t)*u_{0}-\int_{0}^{t}\p_{t}G(t-\tau)*\p_{x}(g(u))(\tau)d\tau-G(0)*\p_{x}(g(u))(t), 
\end{split}
\end{align*}
where $g(u)=\frac{b}{2}u^{2}+\frac{c}{3!}u^{3}$. Moreover, because \eqref{2-2} and \eqref{2-3}, we see that $G(0)*\rho=0$ for any $\rho$. Therefore, we arrive at 
\begin{align}
\label{3-39}
\begin{split} 
\p_{t}u(t)=\p_{t}G(t)*(u_{0}+u_{1})+\p_{t}^{2}G(t)*u_{0}-\int_{0}^{t}\p_{t}G(t-\tau)*\p_{x}(g(u))(\tau)d\tau.
\end{split}
\end{align}
On the other hand, we have from \eqref{3-2}
\begin{align}
\label{3-40}
\begin{split}
\p_{t}\chi(t)=\p_{t}G_{0}(t)*\chi_{0}-\frac{b}{2}\int_{0}^{t}\p_{t}G_{0}(t-\tau)*\p_{x}(\chi^{2})(\tau)d\tau-\frac{b}{2}\p_{x}(\chi^{2})(t), 
\end{split}
\end{align}
where $\chi_{0}(x)=\chi(x, 0)$. Thus, combining \eqref{3-39} and \eqref{3-40}, it follows that 
\begin{align}
\label{3-41}
\begin{split}
\p_{t}(u(t)-\chi(t))&=\p_{t}(G-G_{0})(t)*(u_{0}+u_{1})+\p_{t}G_{0}(t)*(u_{0}+u_{1}-\chi_{0})\\
&\ \ \ +\p_{t}^{2}G(t)*u_{0}-\int_{0}^{t}\p_{t}G(t-\tau)*\p_{x}\left(g(u)-\frac{b}{2}\chi^{2}\right)(\tau)d\tau\\
&\ \ \ -\frac{b}{2}\int_{0}^{t}\p_{t}(G-G_{0})(t-\tau)*\p_{x}(\chi^{2})(\tau)d\tau+\frac{b}{2}\p_{x}(\chi^{2})(t)\\
&\equiv J_{1}+J_{2}+J_{3}+J_{4}+J_{5}+J_{6}. 
\end{split}
\end{align}

Now, we shall evaluate the all terms of the right hand side of \eqref{3-41}. Here and after in this proof, we set $\gamma \equiv \min\{\alpha, \beta\}$. First for $J_{1}$, it follows from \eqref{2-7} that 
\begin{align}
\label{3-42}
\begin{split}
\|\p_{x}^{l}J_{1}(\cdot, t)\|_{L^{p}}&\le CE_{0}^{(s, p)}t^{-\frac{1}{2}(1-\frac{1}{p})-\frac{1+l}{2}}(1+t)^{-\frac{1}{2}} \\
&\le CE_{0}^{(s, p)}(1+t)^{-\frac{3}{2}+\frac{1}{2p}-\frac{l}{2}}, \ t\ge1.
\end{split}
\end{align}
Also, since $\int_{\R}(u_{0}+u_{1}-\chi_{0})dx=0$, we have from \eqref{2-15}
\begin{align}
\label{3-43}
\|\p_{x}^{l}J_{2}(\cdot, t)\|_{L^{p}}\le C\begin{cases}
(1+t)^{-\frac{\gamma}{2}+\frac{1}{2p}-\frac{1+l}{2}}, &t\ge1, \ 1<\gamma<2,\\
(1+t)^{-\frac{3}{2}+\frac{1}{2p}-\frac{l}{2}}\log(2+t), &t\ge1, \ \gamma=2.
\end{cases}
\end{align}
Similarly, from \eqref{2-5}, we obtain 
\begin{align}
\label{3-44}
\|\p_{x}^{l}J_{3}(\cdot, t)\|_{L^{p}}&\le CE_{0}^{(s, p)}(1+t)^{-\frac{3}{2}+\frac{1}{2p}-\frac{l}{2}}, \ t\ge0. 
\end{align}
Next, we evaluate $J_{4}$. From \eqref{2-9}, \eqref{2-10}, \eqref{2-12}, \eqref{3-26} and \eqref{3-27}, we have
\begin{align}
\label{3-45}
\biggl\|\biggl(g(u)-\frac{b}{2}\chi^{2}\biggl)(\cdot, t)\biggl\|_{L^{1}}&\le C\begin{cases}
(1+t)^{-\frac{\gamma}{2}}, &t\ge0, \ 1<\gamma<2,\\
(1+t)^{-1+\e}, &t\ge0, \ \gamma=2,
\end{cases}
\end{align}
\begin{align}
\label{3-46}
\biggl\|\p_{x}^{l}\biggl(g(u)-\frac{b}{2}\chi^{2}\biggl)(\cdot, t)\biggl\|_{L^{p}}&\le C\begin{cases}
(1+t)^{-\frac{\gamma}{2}+\frac{1}{2p}-\frac{1+l}{2}}, &t\ge0, \ 1<\gamma<2,\\
(1+t)^{-\frac{3}{2}+\frac{1}{2p}-\frac{l}{2}+\e}, &t\ge0, \ \gamma=2,
\end{cases}
\end{align}
for $0\le l \le s$. Therefore, we can compute that from \eqref{2-5}, \eqref{3-45} and \eqref{3-46} that 
\begin{align}
\label{3-47}
\begin{split}
\|\p_{x}^{l}J_{4}(\cdot, t)\|_{L^{p}}&\le \int_{0}^{t/2}\biggl\|\p_{t}\p_{x}^{l+1}G(t-\tau)*\biggl(g(u)-\frac{b}{2}\chi^{2}\biggl)(\tau)\biggl\|_{L^{p}}d\tau\\
&\ \ \ +\int_{t/2}^{t}\biggl\|\p_{t}G(t-\tau)*\p_{x}^{l+1}\biggl(g(u)-\frac{b}{2}\chi^{2}\biggl)(\tau)\biggl\|_{L^{p}}d\tau\\
&\le C\int_{0}^{t/2}(1+t-\tau)^{-\frac{3}{2}+\frac{1}{2p}-\frac{l}{2}}\biggl\|\biggl(g(u)-\frac{b}{2}\chi^{2}\biggl)(\cdot, \tau)\biggl\|_{L^{1}}d\tau\\
&\ \ \ +C\int_{t/2}^{t}(1+t-\tau)^{-\frac{1}{2}}\biggl\|\p_{x}^{l+1}\biggl(g(u)-\frac{b}{2}\chi^{2}\biggl)(\cdot, \tau)\biggl\|_{L^{p}}d\tau\\
&\le C\begin{cases}\d \int_{0}^{t/2}(1+t-\tau)^{-\frac{3}{2}+\frac{1}{2p}-\frac{l}{2}}(1+\tau)^{-\frac{\gamma}{2}}d\tau, &1<\gamma<2,\\
\d \int_{0}^{t/2}(1+t-\tau)^{-\frac{3}{2}+\frac{1}{2p}-\frac{l}{2}}(1+\tau)^{-1+\e}d\tau, &\gamma=2
\end{cases}\\
&\ \ \ +C\begin{cases}\d \int_{t/2}^{t}(1+t-\tau)^{-\frac{1}{2}}(1+\tau)^{-1+\frac{1}{2p}-\frac{l}{2}-\frac{\gamma}{2}}d\tau, &1<\gamma<2,\\
\d \int_{t/2}^{t}(1+t-\tau)^{-\frac{1}{2}}(1+\tau)^{-2+\frac{1}{2p}-\frac{l}{2}+\e}d\tau, &\gamma=2
\end{cases}\\
&\le C\begin{cases}(1+t)^{-\frac{\gamma}{2}+\frac{1}{2p}-\frac{1+l}{2}}, &t\ge0, \ 1<\gamma<2, \\
 (1+t)^{-\frac{3}{2}+\frac{1}{2p}-\frac{l}{2}+\e}, &t\ge0, \ \gamma=2.
 \end{cases}
\end{split}
\end{align}
On the other hand, by using \eqref{2-7} and \eqref{2-12}, it follows that 
\begin{align}
\label{3-48}
\begin{split}
\|\p_{x}^{l}J_{5}(\cdot, t)\|_{L^{p}}&\le C\int_{0}^{t/2}\|\p_{t}\p_{x}^{l+1}(G-G_{0})(t-\tau)*(\chi^{2})(\tau)\biggl\|_{L^{p}}d\tau\\
&\ \ \ +C\int_{t/2}^{t}\biggl\|\p_{t}(G-G_{0})(t-\tau)*\p_{x}^{l+1}(\chi^{2})(\tau)\|_{L^{p}}d\tau\\
&\le C\int_{0}^{t/2}(t-\tau)^{-\frac{3}{2}+\frac{1}{2p}-\frac{l}{2}}(1+t-\tau)^{-\frac{1}{2}}(\|\chi^{2}(\cdot, \tau)\|_{L^{1}}+\|\chi^{2}(\cdot, \tau)\|_{W^{l+1, p}})d\tau\\
&\ \ \ +C\int_{t/2}^{t}(t-\tau)^{-\frac{1}{2}}(1+t-\tau)^{-\frac{1}{2}}\|\p_{x}^{l+1}(\chi^{2})(\cdot, \tau)\|_{L^{p}}d\tau\\
&\le CE_{0}^{(s, p)}\int_{0}^{t/2}(t-\tau)^{-\frac{3}{2}+\frac{1}{2p}-\frac{l}{2}}(1+t-\tau)^{-\frac{1}{2}}(1+\tau)^{-\frac{1}{2}}d\tau\\
&\ \ \ +CE_{0}^{(s, p)}\int_{t/2}^{t}(t-\tau)^{-\frac{1}{2}}(1+t-\tau)^{-\frac{1}{2}}(1+\tau)^{-\frac{3}{2}+\frac{1}{2p}-\frac{l}{2}}d\tau\\
&\le CE_{0}^{(s, p)}(1+t)^{-\frac{3}{2}+\frac{1}{2p}-\frac{l}{2}}\log(2+t), \ t\ge1.
\end{split}
\end{align}
Finally, we easily see form \eqref{2-12}
\begin{equation}
\label{3-49}
\|\p_{x}^{l}J_{6}(\cdot, t)\|_{L^{p}}\le CE_{0}^{(s, p)}(1+t)^{-\frac{3}{2}+\frac{1}{2p}-\frac{l}{2}}, \ t\ge0. 
\end{equation}
 Therefore, combining \eqref{3-41} through \eqref{3-44} and \eqref{3-47} through \eqref{3-49}, we arrive at 
 \begin{align}
 \label{3-50}
\|\p_{t}\p_{x}^{l}(u(\cdot, t)-\chi(\cdot, t))\|_{L^{p}}\le C\begin{cases}
(1+t)^{-\frac{\gamma}{2}+\frac{1}{2p}-\frac{1+l}{2}}, &t\ge1, \ 1<\gamma<2,\\
(1+t)^{-\frac{3}{2}+\frac{1}{2p}-\frac{l}{2}+\e}, &t\ge1, \ \gamma=2
\end{cases}
\end{align}
for $0\le l\le s-1$. For $0\le t\le 1$, we obtain from \eqref{2-10} and \eqref{2-12}
\begin{equation}
\label{3-51}
\|\p_{t}^{k}\p_{x}^{l}(u(\cdot, t)-\chi(\cdot, t))\|_{L^{p}}\le \|\p_{t}^{k}\p_{x}^{l}u(\cdot, t)\|_{L^{p}}+\|\p_{t}^{k}\p_{x}^{l}\chi(\cdot, t)\|_{L^{p}}\le CE_{0}^{(s, p)}\le C, 
\end{equation}
where $k=0, 1, 2$ and $l\ge0$ with $0\le k+l\le s$. Summing up \eqref{3-50} and \eqref{3-51}, we get \eqref{1-17} with $k=1$ and $0\le l\le s-1$. 

Next, we shall show \eqref{1-17} with $k=2$ and $0\le l\le s-2$. By using the integration by parts, in the same way to get \eqref{3-39}, we obtain 
\begin{align}
\label{3-52}
\begin{split}
\p_{t}^{2}u(t)&=\p_{t}^{2}G(t)*(u_{0}+u_{1})+\p_{t}^{3}G(t)*u_{0}-\int_{0}^{t/2}\p_{t}^{2}\p_{x}G(t-\tau)*g(u)(\tau)d\tau\\
&\ \ \ -\int_{t/2}^{t}\p_{t}G(t-\tau)*\p_{t}\p_{x}(g(u))(\tau)d\tau-\p_{t}\p_{x}G\left(\frac{t}{2}\right)*\left(g(u)\right)\left(\frac{t}{2}\right),
\end{split}
\end{align}
\begin{align}
\label{3-53}
\begin{split}
\p_{t}^{2}\chi(t)&=\p_{t}^{2}G_{0}(t)*\chi_{0}-\frac{b}{2}\int_{0}^{t/2}\p_{t}^{2}\p_{x}G_{0}(t-\tau)*(\chi^{2})(\tau)d\tau-\frac{b}{2}\p_{t}\p_{x}(\chi^{2})(t)\\
&\ \ \ -\frac{b}{2}\int_{t/2}^{t}\p_{t}G_{0}(t-\tau)*\p_{t}\p_{x}(\chi^{2})(\tau)d\tau-\frac{b}{2}\p_{t}\p_{x}G_{0}\left(\frac{t}{2}\right)*\left(\chi^{2}\right)\left(\frac{t}{2}\right).
\end{split}
\end{align}
Thus, from \eqref{3-52} and \eqref{3-53}, we have 
\begin{align}
\label{3-54}
\begin{split}
&\p_{t}^{2}(u(t)-\chi(t))\\
&=\p_{t}^{2}(G-G_{0})(t)*(u_{0}+u_{1})+\p_{t}^{2}G_{0}(t)*(u_{0}+u_{1}-\chi_{0})+\p_{t}^{3}G(t)*u_{0}\\
&\ \ \ -\int_{0}^{t/2}\p_{t}^{2}\p_{x}G(t-\tau)*\left(g(u)-\frac{b}{2}\chi^{2}\right)(\tau)d\tau-\int_{t/2}^{t}\p_{t}G(t-\tau)*\p_{t}\p_{x}\left(g(u)-\frac{b}{2}\chi^{2}\right)(\tau)d\tau\\
&\ \ \ -\frac{b}{2}\int_{0}^{t/2}\p_{t}^{2}\p_{x}(G-G_{0})(t-\tau)*(\chi^{2})(\tau)d\tau-\frac{b}{2}\int_{t/2}^{t}\p_{t}(G-G_{0})(t-\tau)*\p_{t}\p_{x}(\chi^{2})(\tau)d\tau\\
&\ \ \ +\frac{b}{2}\p_{t}\p_{x}(\chi^{2})-\p_{t}\p_{x}G\left(\frac{t}{2}\right)*\left(g(u)-\frac{b}{2}\chi^{2}\right)\left(\frac{t}{2}\right)-\frac{b}{2}\p_{t}\p_{x}(G-G_{0})\left(\frac{t}{2}\right)*(\chi^{2})\left(\frac{t}{2}\right)\\
&=\tilde{J_{1}}+\tilde{J_{2}}+\tilde{J_{3}}+\tilde{J_{4}}+\tilde{J_{5}}+\tilde{J_{6}}+\tilde{J_{7}}+\tilde{J_{8}}+\tilde{J_{9}}+\tilde{J_{10}}.
\end{split}
\end{align}
Similarly as \eqref{3-46}, we note that 
\begin{align}
\label{3-55}
\biggl\|\p_{t}\p_{x}^{l}\biggl(g(u)-\frac{b}{2}\chi^{2}\biggl)(\cdot, t)\biggl\|_{L^{p}}&\le C\begin{cases}
(1+t)^{-\frac{\gamma}{2}+\frac{1}{2p}-\frac{l}{2}-1}, &t\ge0, \ 1<\gamma<2,\\
(1+t)^{-2+\frac{1}{2p}-\frac{l}{2}+\e}, &t\ge0, \ \gamma=2
\end{cases}
\end{align}
for $0\le l \le s-1$. By using the same argument given in the above paragraph, we have the following estimates:
\begin{equation*}
\|\p_{x}^{l}\tilde{J_{1}}(\cdot, t)\|_{L^{p}}+\|\p_{x}^{l}\tilde{J_{3}}(\cdot, t)\|_{L^{p}}+\|\p_{x}^{l}\tilde{J_{6}}(\cdot, t)\|_{L^{p}}+\|\p_{x}^{l}\tilde{J_{8}}(\cdot, t)\|_{L^{p}}\le C(1+t)^{-2+\frac{1}{2p}-\frac{l}{2}}, \ t\ge1, 
\end{equation*}
\begin{align*}
\|\p_{x}^{l}\tilde{J_{2}}(\cdot, t)\|_{L^{p}}&\le C\begin{cases}
(1+t)^{-\frac{\gamma}{2}+\frac{1}{2p}-\frac{l}{2}-1}, &t\ge1, \ 1<\gamma<2,\\
(1+t)^{-2+\frac{1}{2p}-\frac{l}{2}}\log(2+t), &t\ge1, \ \gamma=2,
\end{cases}
\end{align*}
\begin{align*}
\|\p_{x}^{l}\tilde{J_{4}}(\cdot, t)\|_{L^{p}}+\|\p_{x}^{l}\tilde{J_{5}}(\cdot, t)\|_{L^{p}}&\le C\begin{cases}
(1+t)^{-\frac{\gamma}{2}+\frac{1}{2p}-\frac{l}{2}-1}, &t\ge0, \ 1<\gamma<2,\\
(1+t)^{-2+\frac{1}{2p}-\frac{l}{2}+\e}, &t\ge0, \ \gamma=2
\end{cases}
\end{align*}
and
\begin{equation*}
\|\p_{x}^{l}\tilde{J_{7}}(\cdot, t)\|_{L^{p}}\le C(1+t)^{-2+\frac{1}{2p}-\frac{l}{2}}\log(2+t), \ t\ge1.
\end{equation*}
On the other hand for $\tilde{J_{9}}$, from \eqref{2-5} and \eqref{3-45}, we get
\begin{align*}
\begin{split}
\|\p_{x}^{l}\tilde{J_{9}}(\cdot, t)\|_{L^{p}}&\le C(1+t)^{-\frac{3}{2}+\frac{1}{2p}-\frac{l}{2}}\biggl\|\left(g(u)-\frac{b}{2}\chi^{2}\right)\left(\cdot, \frac{t}{2}\right)\biggl\|_{L^{1}}\\
&\le C\begin{cases}
(1+t)^{-\frac{\gamma}{2}+\frac{1}{2p}-\frac{l}{2}-\frac{3}{2}}, &t\ge0, \ 1<\gamma<2,\\
(1+t)^{-\frac{5}{2}+\frac{1}{2p}-\frac{l}{2}+\e}, &t\ge0, \ \gamma=2.
\end{cases}
\end{split}
\end{align*}
For $\tilde{J_{10}}$, it follows that from \eqref{2-7} and \eqref{2-12}
\begin{align*}
\begin{split}
\|\p_{x}^{l}\tilde{J_{10}}(\cdot, t)\|_{L^{p}}&\le Ct^{-\frac{3}{2}+\frac{1}{2p}-\frac{l}{2}}(1+t)^{-\frac{1}{2}}\left(\biggl\|\chi^{2}\left(\cdot, \frac{t}{2}\right)\biggl\|_{L^{1}}+\ \biggl\|\chi^{2}\left(\cdot, \frac{t}{2}\right)\biggl\|_{W^{l+1,p}}\right)\\
&\le C(1+t)^{-\frac{5}{2}+\frac{1}{2p}-\frac{l}{2}}, \ t\ge1. 
\end{split}
\end{align*}
Therefore, combining all the above estimates, we can derive  
 \begin{align}
 \label{3-56}
\|\p_{t}^{2}\p_{x}^{l}(u(\cdot, t)-\chi(\cdot, t))\|_{L^{p}}\le C\begin{cases}
(1+t)^{-\frac{\gamma}{2}+\frac{1}{2p}-\frac{l}{2}-1}, &t\ge1, \ 1<\gamma<2,\\
(1+t)^{-2+\frac{1}{2p}-\frac{l}{2}+\e}, &t\ge1, \ \gamma=2
\end{cases}
\end{align}
for $0\le l\le s-2$. Thus, summing up \eqref{3-56} and \eqref{3-51}, we can prove \eqref{1-17} with $k=2$. This completes the proof.
\end{proof}

Finally in this section, we give the additional decay estimate for $u-\chi$. From the original equations \eqref{1-2} and \eqref{1-6}, we see that 
\begin{equation*}
(\p_{t}+a\p_{x})(u-\chi)=(-\p_{t}^{2}+\p_{x}^{2})(u-\chi)-\frac{b}{2}\p_{x}(u^{2}-\chi^{2})-\frac{c}{3!}\p_{x}(u^{3})-(\p_{t}-a\p_{x})(\p_{t}+a\p_{x})\chi.
\end{equation*}
By virtue of this relation, we have the following estimate:
\begin{cor}\label{cor.add}
Assume the same conditions on $u_{0}$ and $u_{1}$ in Theorem~\ref{main1} are valid. Then, for any $\e>0$, the estimate 
\begin{align}
\label{3-57}
\|\p_{x}^{l}(\p_{t}+a\p_{x})((u-\chi)(\cdot, t))\|_{L^{p}}\le C\begin{cases}
(1+t)^{-\frac{\min\{\alpha, \beta\}}{2}+\frac{1}{2p}-\frac{l}{2}-1}, &t\ge0, \ 1<\min\{\alpha, \beta\}<2,\\
(1+t)^{-2+\frac{1}{2p}-\frac{l}{2}+\e}, &t\ge0, \ \min\{\alpha, \beta\}=2
\end{cases}
\end{align}
holds for $0\le l\le s-2$. 
\end{cor} 
\noindent
We will use this estimate in the proof of Theorem~\ref{main2}.

\section{Proof of Theorem 1.2 for $1<\min\{\alpha, \beta\}< 2$}

In this section, we shall prove Theorem~\ref{main2} in the case of $1<\min\{\alpha, \beta\}<2$. Namely, our purpose of this section is to derive \eqref{1-19} and \eqref{1-23}. First, we set 
\begin{equation}
\label{4-1}
\psi(x, t)\equiv u(x, t)+u_{t}(x, t)-\chi(x, t), \ \ \psi_{0}(x)\equiv u_{0}(x)+u_{1}(x)-\chi_{0}(x), 
\end{equation}
where $\chi_{0}$ is defined by $\chi_{0}(x)=\chi(x, 0)$. Then we have the following initial value problem:
\begin{align}
\label{4-2}
\begin{split}
\psi_{t}+a\psi_{x}+(b\chi \psi)_{x}-\mu \psi_{xx}&=a\p_{x}(\p_{t}+a\p_{x})(u-\chi)+a\p_{x}(\p_{t}+a\p_{x})\chi\\
&\ \ \ -\frac{b}{2}\p_{x}((u-\chi)^{2})-\frac{c}{3!}\p_{x}(u^{3})+\p_{x}(b\chi u_{t}-\mu u_{tx}), \ x\in \R, \ t>0,\\
\psi(x, 0)&=\psi_{0}(x), \ x\in \R.\\
\end{split}
\end{align}
Therefore, from Lemma~\ref{RF}, we obtain 
\begin{align}
\label{4-3}
\begin{split}
\psi(x, t)=&\ U[\psi_{0}](x, t, 0)\\
&+a\int_{0}^{t}U[\p_{x}(\p_{t}+a\p_{x})(u-\chi)(\tau)](x, t, \tau)d\tau+a\int_{0}^{t}U[\p_{x}(\p_{t}+a\p_{x})\chi(\tau)](x, t, \tau)d\tau\\
&-\frac{b}{2}\int_{0}^{t}U[\p_{x}((u-\chi)^{2})(\tau)](x, t, \tau)d\tau-\frac{c}{3!}\int_{0}^{t}U[\p_{x}(u^{3})(\tau)](x, t, \tau)d\tau\\
&+b\int_{0}^{t}U[\p_{x}(u_{t}\chi)(\tau)](x, t, \tau)d\tau-\mu \int_{0}^{t}U[\p_{x}(u_{tx})(\tau)](x, t, \tau)d\tau.
\end{split}
\end{align}
For the first term of the right hand side in the above equation \eqref{4-3}, we have the following asymptotic formula. The proof is given by almost the same argument as in Proposition 4.1 in \cite{F19-2}. 
\begin{prop}\label{AF}
Assume the same conditions on $u_{0}$ and $u_{1}$ in Theorem~\ref{main2} are valid. Then we have 
\begin{align}
\label{4-4}
&\lim_{t\to \infty}(1+t)^{\frac{\min\{\alpha, \beta\}}{2}}\|U[\psi_{0}](\cdot, t, 0)-Z(\cdot, t)\|_{L^{\infty}}=0, \ 1<\min\{\alpha, \beta\}<2, \\
\label{4-5}
&\lim_{t\to \infty}\frac{(1+t)}{\log(1+t)}\|U[\psi_{0}](\cdot, t, 0)-Z(\cdot, t)\|_{L^{\infty}}=0, \ \min\{\alpha, \beta\}=2,  
\end{align}
where $Z(x, t)$ is defined by \eqref{1-21}. 
\end{prop}
\begin{proof}
From the definition of $U$ given by \eqref{2-26} and $\eta_{0}(x)=\eta(x, 0)$, we have 
\begin{align}
\label{4-6}
\begin{split}
U[\psi_{0}](x, t, 0)=&\int_{\R}\p_{x}(G_{0}(x-y, t)\eta(x, t))\eta_{0}(y)^{-1}\biggl(\int_{-\infty}^{y}(u_{0}(\xi)+u_{1}(\xi)-\chi_{0}(\xi))d\xi\biggl) dy\\
=&\int_{\R}\p_{x}(G_{0}(x-y, t)\eta(x, t))z_{0}(y)dy, 
\end{split}
\end{align} 
where $z_{0}(y)$ is defined by \eqref{1-18}. First, we shall check the following estimate: 
\begin{equation}
\label{4-7}
|z_{0}(x)|\le C(1+|x|)^{-(\min\{\alpha, \beta\}-1)}, \ \ x\in \R. 
\end{equation}
If $x<0$, from \eqref{2-22}, \eqref{1-3}, \eqref{1-4} and \eqref{1-5}, we have
\begin{align*}
|z_{0}(x)|&\le C\int_{-\infty}^{x}(|u_{0}(y)|+|u_{1}(y)|+|\chi_{0}(y)|)dy \\
&\le C\int_{-\infty}^{x}(1+|y|)^{-\min\{\alpha, \beta\}}dy+C\int_{-\infty}^{x}(1+|y|)^{-N}dy \ \ (\forall N\ge0) \\
&\le C\int_{-\infty}^{x}(1-y)^{-\min\{\alpha, \beta\}}dy\le C(1-x)^{-(\min\{\alpha, \beta\}-1)}=C(1+|x|)^{-(\min\{\alpha, \beta\}-1)}. 
\end{align*}
On the other hand, since $\int_{\R}(u_{0}(x)+u_{1}(x))dx=\int_{\R}\chi_{0}(x)dx=M$, if $x>0$, similarly we have
\begin{align*}
|z_{0}(x)|&\le C\biggl|\int_{-\infty}^{x}(u_{0}(y)+u_{1}(y)-\chi_{0}(y))dy\biggl| \\
&=C\biggl|\int_{-\infty}^{x}(u_{0}(y)+u_{1}(y))dy-\int_{\R}(u_{0}(y)+u_{1}(y))dy-\int_{-\infty}^{x}\chi_{0}(y)dy+\int_{\R}\chi_{0}(y)dy\biggl| \\
&\le C\biggl|\int_{x}^{\infty}(u_{0}(y)+u_{1}(y))dy\biggl|+\biggl|\int_{x}^{\infty}\chi_{0}(y)dy\biggl| \\
&\le C\int_{x}^{\infty}(1+|y|)^{-\min\{\alpha, \beta\}}dy+C\int_{x}^{\infty}(1+|y|)^{-N}dy \ \ (\forall N\ge0) \\
&\le C\int_{x}^{\infty}(1+y)^{-\min\{\alpha, \beta\}}dy\le C(1+x)^{-(\min\{\alpha, \beta\}-1)}=C(1+|x|)^{-(\min\{\alpha, \beta\}-1)}. 
\end{align*}
Therefore we get \eqref{4-7}. Thus, we obtain the boundedness of $(1+|y|)^{\min\{\alpha, \beta\}-1}z_{0}(y)$. Moreover, from the assumption on $z_{0}(y)$, for any $\e>0$ there is a constant $R=R(\e)>0$ such that
 \begin{align*}
&|z_{0}(y)-c_{\alpha, \beta}^{+}(1+|y|)^{-(\min\{\alpha, \beta\}-1)}|\le \e(1+|y|)^{-(\min\{\alpha, \beta\}-1)}, \ y\ge R, \\
&|z_{0}(y)-c_{\alpha, \beta}^{-}(1+|y|)^{-(\min\{\alpha, \beta\}-1)}|\le \e(1+|y|)^{-(\min\{\alpha, \beta\}-1)}, \ y\le -R. 
 \end{align*}
 From \eqref{1-21} and \eqref{4-6}, we have the following estimate
 \begin{align}
 \label{4-8}
 \begin{split}
&|U[\psi_{0}](x, t, 0)-Z(x, t)|\\
 &\le \int_{\R}|\p_{x}(G_{0}(x-y, t)\eta(x, t))||z_{0}(y)-c_{\alpha, \beta}(y)(1+|y|)^{-(\min\{\alpha, \beta\}-1)}|dy\\
 &\le \int_{|y|\le R}|\p_{x}(G_{0}(x-y, t)\eta(x, t))||z_{0}(y)-c_{\alpha, \beta}(y)(1+|y|)^{-(\min\{\alpha, \beta\}-1)}|dy\\
 &\ \ \ \ +\e \int_{|y|\ge R}|\p_{x}(G_{0}(x-y, t)\eta(x, t))|(1+|y|)^{-(\min\{\alpha, \beta\}-1)}dy \\
 &\le C\sum_{n=0}^{1}\|\p_{x}^{1-n}\eta(\cdot, t)\|_{L^{\infty}}\|\p_{x}^{n}G_{0}(\cdot, t)\|_{L^{\infty}}\int_{|y|\le R}|z_{0}(y)-c_{\alpha, \beta}(y)(1+|y|)^{-(\min\{\alpha, \beta\}-1)}|dy\\
 &\ \ \ +\e C\sum_{n=0}^{1}\|\p_{x}^{1-n}\eta(\cdot, t)\|_{L^{\infty}}\int_{\R}|\p_{x}^{n}G_{0}(x-y, t)|(1+|y|)^{-(\min\{\alpha, \beta\}-1)}dy.\\
 \end{split}
 \end{align}
 For the integral in the last term of the right hand side of \eqref{4-8}, we can estimate it as follows 
 \begin{align}
 \label{4-9}
\begin{split}
&\int_{\R} |\p_{x}^{n}G_{0}(x-y, t)|(1+|y|)^{-(\min\{\alpha, \beta\}-1)}dy \\
&=\biggl(\int_{|y|\ge \sqrt{1+t}-1}+\int_{|y|\le \sqrt{1+t}-1}\biggl)|\p_{x}^{n}G_{0}(x-y, t)|(1+|y|)^{-(\min\{\alpha, \beta\}-1)}dy \\
&\le \biggl(\sup_{|y|\ge \sqrt{1+t}-1}(1+|y|)^{-(\min\{\alpha, \beta\}-1)}\biggl)\int_{|y|\ge \sqrt{1+t}-1}|\p_{x}^{n}G_{0}(x-y, t)|dy\\
&\ \ \ \ +\biggl(\sup_{|y|\le \sqrt{1+t}-1}|\p_{x}^{n}G_{0}(x-y, t)|\biggl)\int_{|y|\le \sqrt{1+t}-1}(1+|y|)^{-(\min\{\alpha, \beta\}-1)}dy\\
&\le (1+t)^{-\frac{\min\{\alpha, \beta\}-1}{2}}\|\p_{x}^{n}G_{0}(\cdot, t)\|_{L^{1}}+\|\p_{x}^{n}G_{0}(\cdot, t)\|_{L^{\infty}}\int_{|y|\le \sqrt{1+t}-1}(1+|y|)^{-(\min\{\alpha, \beta\}-1)}dy\\
&\le C(1+t)^{-\frac{\min\{\alpha, \beta\}-1}{2}-\frac{n}{2}}+Ct^{-\frac{1+n}{2}}\int_{0}^{\sqrt{1+t}-1}(1+y)^{-(\min\{\alpha, \beta\}-1)}dy \\
&\le C(1+t)^{-\frac{\min\{\alpha, \beta\}-1}{2}-\frac{n}{2}}+Ct^{-\frac{1+n}{2}}\begin{cases}
(1+t)^{-\frac{\min\{\alpha, \beta\}-1}{2}+\frac{1}{2}}, &1<\min\{\alpha, \beta\}<2,  \\
\log(1+t), &\min\{\alpha, \beta\}=2
\end{cases}\\
&\le C\begin{cases}
(1+t)^{-\frac{\min\{\alpha, \beta\}-1}{2}-\frac{n}{2}}, &t\ge1, \ 1<\min\{\alpha, \beta\}<2, \\
(1+t)^{-\frac{1+n}{2}}\log(1+t), &t\ge1, \ \min\{\alpha, \beta\}=2.
\end{cases}
\end{split}
\end{align} 
Therefore, by using \eqref{4-8}, Lemma~\ref{eta.decay}, Lemma~\ref{heat.decay} and \eqref{4-9}, we get 
  \begin{align*}
 \begin{split}
 &\|U[\psi_{0}](\cdot, t, 0)-Z(\cdot, t)\|_{L^{\infty}}\\
 &\le C(1+t)^{-1}+\e C\begin{cases}
(1+t)^{-\frac{\min\{\alpha, \beta\}}{2}}, &t\ge1, \ 1<\min\{\alpha, \beta\}<2, \\
(1+t)^{-1}\log(1+t), &t\ge1, \ \min\{\alpha, \beta\}=2.\end{cases}
  \end{split}
 \end{align*}
 Thus, we obtain 
 \begin{align*}
&\limsup_{t\to \infty}(1+t)^{\frac{\min\{\alpha, \beta\}}{2}}\|U[\psi_{0}](\cdot, t, 0)-Z(\cdot, t)\|_{L^{\infty}}\le\e C, \ 1<\min\{\alpha, \beta\}<2, \\
&\limsup_{t\to \infty}\frac{(1+t)}{\log(1+t)}\|U[\psi_{0}](\cdot, t, 0)-Z(\cdot, t)\|_{L^{\infty}} \le\e C, \ \min\{\alpha, \beta\}=2.  
\end{align*}
Therefore, we get \eqref{4-4} and \eqref{4-5}, because $\e>0$ can be chosen arbitrarily small. 
\end{proof}

\newpage
\begin{proof}[\rm{\bf{End of the proof of Theorem 1.2 for $1<\min\{\alpha, \beta\}<2$}}]

We shall prove \eqref{1-19} and \eqref{1-23}. By using \eqref{1-21} and \eqref{4-3}, we have 
\begin{align}
\label{4-10}
\begin{split}
&u(x, t)-\chi(x, t)-Z(x, t)\\
&=U[\psi_{0}](x, t, 0)-Z(x, t)-u_{t}(x, t)\\
&\ \ \ \ +a\int_{0}^{t}U[\p_{x}(\p_{t}+a\p_{x})(u-\chi)(\tau)](x, t, \tau)d\tau+a\int_{0}^{t}U[\p_{x}(\p_{t}+a\p_{x})\chi(\tau)](x, t, \tau)d\tau\\
&\ \ \ \ -\frac{b}{2}\int_{0}^{t}U[\p_{x}((u-\chi)^{2})(\tau)](x, t, \tau)d\tau-\frac{c}{3!}\int_{0}^{t}U[\p_{x}(u^{3})(\tau)](x, t, \tau)d\tau\\
&\ \ \ \ +b\int_{0}^{t}U[\p_{x}(u_{t}\chi)(\tau)](x, t, \tau)d\tau-\mu\int_{0}^{t}U[\p_{x}(u_{tx})(\tau)](x, t, \tau)d\tau\\
&\equiv U[\psi_{0}](x, t, 0)-Z(x, t)-u_{t}(x, t)+J_{1}+J_{2}+J_{3}+J_{4}+J_{5}+J_{6}.
\end{split}
\end{align}

We shall evaluate $J_{1}$, $J_{2}$, $J_{3}$, $J_{4}$, $J_{5}$ and $J_{6}$. First for $J_{1}$, from Lemma~\ref{N.decay} and \eqref{3-57}, we have 
\begin{align}
\label{4-11}
\begin{split}
\|J_{1}(\cdot, t)\|_{L^{\infty}}&\le C\sum_{n=0}^{1}(1+t)^{-\frac{1}{2}+\frac{n}{2}}\biggl(\int_{0}^{t/2}(t-\tau)^{-\frac{1}{2}-\frac{n}{2}}\| (\p_{t}+a\p_{x})((u-\chi)(\cdot, \tau))\|_{L^{1}}d\tau\\
&\ \ \ \ +\int_{t/2}^{t}(t-\tau)^{-\frac{1}{4}-\frac{n}{2}}\| (\p_{t}+a\p_{x})((u-\chi)(\cdot, \tau))\|_{L^{2}}d\tau \biggl) \\
&\le C\sum_{n=0}^{1}(1+t)^{-\frac{1}{2}+\frac{n}{2}}\biggl(\int_{0}^{t/2}(t-\tau)^{-\frac{1}{2}-\frac{n}{2}}(1+\tau)^{-\frac{\min\{\alpha, \beta\}+1}{2}}d\tau\\
&\ \ \ \ +\int_{t/2}^{t}(t-\tau)^{-\frac{1}{4}-\frac{n}{2}}(1+\tau)^{-\frac{\min\{\alpha, \beta\}}{2}-\frac{3}{4}}d\tau \biggl) \\
&\le C(1+t)^{-1}, \ t\ge1.
\end{split}
\end{align}
On the other hand, we have from \eqref{2-12} that 
\begin{align}
\label{4-12}
\begin{split}
\|J_{2}(\cdot, t)\|_{L^{\infty}}&\le C\sum_{n=0}^{1}(1+t)^{-\frac{1}{2}+\frac{n}{2}}\biggl(\int_{0}^{t/2}(t-\tau)^{-\frac{1}{2}-\frac{n}{2}}\| (\p_{t}+a\p_{x})\chi(\cdot, \tau)\|_{L^{1}}d\tau\\
&\ \ \ \ +\int_{t/2}^{t}(t-\tau)^{-\frac{n}{2}}\| (\p_{t}+a\p_{x})\chi(\cdot, \tau)\|_{L^{\infty}}d\tau \biggl) \\
&\le C\sum_{n=0}^{1}(1+t)^{-\frac{1}{2}+\frac{n}{2}}\biggl(\int_{0}^{t/2}(t-\tau)^{-\frac{1}{2}-\frac{n}{2}}(1+\tau)^{-1}d\tau+\int_{t/2}^{t}(t-\tau)^{-\frac{n}{2}}(1+\tau)^{-\frac{3}{2}}d\tau \biggl) \\
&\le C(1+t)^{-1}\log(1+t), \ t\ge1.
\end{split}
\end{align}
We note that $J_{1}$ and $J_{2}$ do not appear if $a=0$. Next, from Lemma~\ref{N.decay} and \eqref{1-16}, we obtain 
\begin{align}
\label{4-13}
\begin{split}
\|J_{3}(\cdot, t)\|_{L^{\infty}}&\le C\sum_{n=0}^{1}(1+t)^{-\frac{1}{2}+\frac{n}{2}}\biggl(\int_{0}^{t/2}(t-\tau)^{-\frac{1}{2}-\frac{n}{2}}\| (u-\chi)^{2}(\cdot, \tau)\|_{L^{1}}d\tau\\
&\ \ \ \ +\int_{t/2}^{t}(t-\tau)^{-\frac{n}{2}}\| (u-\chi)^{2}(\cdot, \tau)\|_{L^{\infty}}d\tau \biggl) \\
&\le C\sum_{n=0}^{1}(1+t)^{-\frac{1}{2}+\frac{n}{2}}\biggl(\int_{0}^{t/2}(t-\tau)^{-\frac{1}{2}-\frac{n}{2}}(1+\tau)^{-\min\{\alpha, \beta\}+\frac{1}{2}}d\tau\\
&\ \ \ \ +\int_{t/2}^{t}(t-\tau)^{-\frac{n}{2}}(1+\tau)^{-\min\{\alpha, \beta\}}d\tau \biggl) \\
&\le C
\begin{cases}
(1+t)^{-1}, &t\ge1, \ 3/2<\min\{\alpha, \beta\}<2, \\
(1+t)^{-1}\log(1+t), &t\ge1, \ \min\{\alpha, \beta\}=3/2, \\
(1+t)^{-\min\{\alpha, \beta\}+1/2}, &t\ge1, \ 1<\min\{\alpha, \beta\}<3/2.
\end{cases}
\end{split}
\end{align}
For $J_{4}$, in the same way to get \eqref{4-11}, from Lemma~\ref{N.decay} and \eqref{2-9}, we get 
\begin{align}
\label{4-14}
\begin{split}
\|J_{4}(\cdot, t)\|_{L^{\infty}}&\le C\sum_{n=0}^{1}(1+t)^{-\frac{1}{2}+\frac{n}{2}}\\
&\ \ \ \ \times \biggl(\int_{0}^{t/2}(t-\tau)^{-\frac{1}{2}-\frac{n}{2}}\| u^{3}(\cdot, \tau)\|_{L^{1}}d\tau+\int_{t/2}^{t}(t-\tau)^{-\frac{n}{2}}\| u^{3}(\cdot, \tau)\|_{L^{\infty}}d\tau \biggl) \\
&\le C\sum_{n=0}^{1}(1+t)^{-\frac{1}{2}+\frac{n}{2}}\\
&\ \ \ \ \times \biggl(\int_{0}^{t/2}(t-\tau)^{-\frac{1}{2}-\frac{n}{2}}(1+\tau)^{-1}d\tau+\int_{t/2}^{t}(t-\tau)^{-\frac{n}{2}}(1+\tau)^{-\frac{3}{2}}d\tau \biggl) \\
&\le C(1+t)^{-1}\log(1+t), \ t\ge1. 
\end{split}
\end{align}
We note that $J_{4}$ does not appear if $c=0$. Also, from Lemma~\ref{N.decay}, Gagliardo-Nirenberg inequality, \eqref{2-10} and \eqref{2-12}, we have the estimate for $J_{5}$ as follows.  
\begin{align}
\label{4-15}
\begin{split}
\|J_{5}(\cdot, t)\|_{L^{\infty}}&\le C\sum_{n=0}^{1}(1+t)^{-\frac{1}{2}+\frac{n}{2}}\biggl(\int_{0}^{t/2}(t-\tau)^{-\frac{1}{2}-\frac{n}{2}}\| u_{t}(\cdot, \tau)\|_{L^{\infty}}\|\chi(\cdot, \tau)\|_{L^{1}}d\tau\\
&\ \ \ \ +\int_{t/2}^{t}(t-\tau)^{-\frac{n}{2}}\|u_{t}(\cdot, \tau)\|_{L^{\infty}}\|\chi(\cdot, \tau)\|_{L^{\infty}}d\tau \biggl) \\
&\le C\sum_{n=0}^{1}(1+t)^{-\frac{1}{2}+\frac{n}{2}}\biggl(\int_{0}^{t/2}(t-\tau)^{-\frac{1}{2}-\frac{n}{2}}(1+\tau)^{-1}d\tau+\int_{t/2}^{t}(t-\tau)^{-\frac{n}{2}}(1+\tau)^{-\frac{3}{2}}d\tau \biggl) \\
&\le C(1+t)^{-1}\log(1+t), \ t\ge1. 
\end{split}
\end{align}
Finally, we evaluate $J_{6}$. Similarly as before, it follows that 
\begin{align}
\label{4-16}
\begin{split}
\|J_{6}(\cdot, t)\|_{L^{\infty}}&\le C\sum_{n=0}^{1}(1+t)^{-\frac{1}{2}+\frac{n}{2}}\\
&\ \ \ \times \biggl(\int_{0}^{t/2}(t-\tau)^{-\frac{1}{2}-\frac{n}{2}}\| u_{tx}(\cdot, \tau)\|_{L^{1}}d\tau+\int_{t/2}^{t}(t-\tau)^{-\frac{1}{4}-\frac{n}{2}}\|u_{tx}(\cdot, \tau)\|_{L^{2}}d\tau \biggl) \\
&\le C\sum_{n=0}^{1}(1+t)^{-\frac{1}{2}+\frac{n}{2}}\biggl(\int_{0}^{t/2}(t-\tau)^{-\frac{1}{2}-\frac{n}{2}}(1+\tau)^{-1}d\tau+\int_{t/2}^{t}(t-\tau)^{-\frac{1}{4}-\frac{n}{2}}(1+\tau)^{-\frac{5}{4}}d\tau \biggl) \\
&\le C(1+t)^{-1}\log(1+t), \ t\ge1. 
\end{split}
\end{align}
By using \eqref{4-10}, \eqref{2-10}, Gagliardo-Nirenberg inequality and \eqref{4-11} through \eqref{4-16}, we have
\begin{align*}
\|u(\cdot, t)-\chi(\cdot, t)-Z(\cdot, t)\|_{L^{\infty}}&\le \|U[\psi_{0}](\cdot, t, 0)-Z(\cdot, t)\|_{L^{\infty}}+C(1+t)^{-1}+C(1+t)^{-1}\log(1+t)\\
&\ \ \ \ +C\begin{cases}
(1+t)^{-1}, &t\ge1, \ \ 3/2<\min\{\alpha, \beta\}<2,  \\
(1+t)^{-1}\log(1+t), &t\ge1, \ \ \min\{\alpha, \beta\}=3/2,  \\
(1+t)^{-\min\{\alpha, \beta\}+1/2}, &t\ge1, \ \ 1<\min\{\alpha, \beta\}<3/2.
\end{cases}
\end{align*}
Therefore, from \eqref{4-4}, we finally obtain 
\begin{equation*}
\limsup_{t\to \infty}(1+t)^{\frac{\min\{\alpha, \beta\}}{2}}\|u(\cdot, t)-\chi(\cdot, t)-Z(\cdot, t)\|_{L^{\infty}}=0.
\end{equation*}
Thus we completes the proof of \eqref{1-19}. 

In the rest of this proof, we shall prove \eqref{1-23}. First, in the same way to get \eqref{4-8}, by using Lemma~\ref{eta.decay} and \eqref{4-9}, we can derive the upper bound estimate of \eqref{1-23} as follows.
 \begin{align}
 \label{4-17}
 \begin{split}
|Z(x, t)|&\le C\max\{|c_{\alpha, \beta}^{+}|, |c_{\alpha, \beta}^{-}|\}\sum_{n=0}^{1}\|\p_{x}^{1-n}\eta(\cdot, t)\|_{L^{\infty}}\int_{\R}|\p_{x}^{n}G_{0}(x-y, t)|(1+|y|)^{-(\min\{\alpha, \beta\}-1)}dy\\
&\le C\max\{|c_{\alpha, \beta}^{+}|, |c_{\alpha, \beta}^{-}|\}(1+t)^{-\frac{1}{2}+\frac{n}{2}}\begin{cases}
(1+t)^{-\frac{\min\{\alpha, \beta\}-1}{2}+\frac{n}{2}}, &t\ge1, \ \ 1<\min\{\alpha, \beta\}<2, \\
(1+t)^{-\frac{1+n}{2}}\log(1+t), &t\ge1, \ \ \min\{\alpha, \beta\}=2
\end{cases}\\
&\le C\max\{|c_{\alpha, \beta}^{+}|, |c_{\alpha, \beta}^{-}|\}\begin{cases}
(1+t)^{-\frac{\min\{\alpha, \beta\}}{2}}, &t\ge1, \ \ 1<\min\{\alpha, \beta\}<2, \\
(1+t)^{-1}\log(1+t), &t\ge1, \ \ \min\{\alpha, \beta\}=2.
\end{cases}
 \end{split}
 \end{align}
Next, we shall prove the lower bound estimate of \eqref{1-23}. For the simplicity, we set $\gamma \equiv \min\{\alpha, \beta\}$. First, we take $x=at$ in \eqref{1-21}, then we have from \eqref{1-4} and \eqref{1-21} 
\begin{align}
\label{4-18}
\begin{split}
Z(at, t)&=\int_{\R}c_{\alpha, \beta}(y)\biggl((\p_{y}G_{0})(at-y, t)\eta(at, t)\\
&\ \ \ +G_{0}(at-y, t)\frac{b}{2\mu}\chi(at, t)\eta(at, t)\biggl)(1+|y|)^{-(\gamma-1)}dy\\
&=c_{\alpha, \beta}^{+}\eta(at, t)\int_{0}^{\infty}(\p_{y}G_{0})(at-y, t)(1+y)^{-(\gamma-1)}dy\\
&\ \ \ +c_{\alpha, \beta}^{-}\eta(at, t)\int_{-\infty}^{0}(\p_{y}G_{0})(at-y, t)(1-y)^{-(\gamma-1)}dy\\
&\ \ \ +\frac{bc_{\alpha, \beta}^{+}}{2\mu}\eta(at, t)\chi(at, t)\int_{0}^{\infty}G_{0}(at-y, t)(1+y)^{-(\gamma-1)}dy \\
&\ \ \ +\frac{bc_{\alpha, \beta}^{-}}{2\mu}\eta(at, t)\chi(at, t)\int_{-\infty}^{0}G_{0}(at-y, t)(1-y)^{-(\gamma-1)}dy\\
&=\frac{(c_{\alpha, \beta}^{+}-c_{\alpha, \beta}^{-})}{4\sqrt{\pi}\mu^{3/2}}\eta(at, t)t^{-\frac{3}{2}}\int_{0}^{\infty}e^{-\frac{y^{2}}{4\mu t}}y(1+y)^{-(\gamma-1)}dy\\
&\ \ \ +\frac{b(c_{\alpha, \beta}^{+}+c_{\alpha, \beta}^{-})}{4\sqrt{\pi}\mu^{3/2}}\eta(at, t)\chi(at, t)t^{-\frac{1}{2}}\int_{0}^{\infty}e^{-\frac{y^{2}}{4\mu t}}(1+y)^{-(\gamma-1)}dy \\
&\equiv L_{1}(t)+L_{2}(t). 
\end{split}
\end{align}
From the mean value theorem, there exists $\theta_{j} \in (0, 1)$ such that 
\begin{equation}
\label{4-19}
(1+y)^{-(\gamma-j)}-y^{-(\gamma-j)}=-(\gamma-j)(y+\theta_{j})^{-(\gamma-j+1)}.
\end{equation}
Therefore, since 
\begin{equation}
\label{4-20}
\int_{0}^{\infty}e^{-\frac{y^{2}}{4\mu t}}y^{j-\gamma}dy=2^{j-\gamma}(\mu t)^{\frac{j+1-\gamma}{2}}\Gamma \biggl(\frac{j+1-\gamma}{2}\biggl), \ \ j\ge1,   
\end{equation}
it follows that 
\begin{align}
\label{4-21}
\begin{split}
L_{1}(t)
&=\frac{(c_{\alpha, \beta}^{+}-c_{\alpha, \beta}^{-})}{4\sqrt{\pi}\mu^{3/2}}\eta(at, t)t^{-\frac{3}{2}}\\
&\ \ \ \ \times \biggl(\int_{0}^{\infty}e^{-\frac{y^{2}}{4\mu t}}y^{2-\gamma}dy+\int_{0}^{\infty}e^{-\frac{y^{2}}{4\mu t}}y((1+y)^{-(\gamma-1)}-y^{-(\gamma-1)})dy\biggl)\\
&=\frac{(c_{\alpha, \beta}^{+}-c_{\alpha, \beta}^{-})}{4\sqrt{\pi}\mu^{3/2}}\eta(at, t)t^{-\frac{3}{2}}\\
&\ \ \ \ \times \biggl(2^{2-\gamma}(\mu t)^{\frac{3-\gamma}{2}}\Gamma\biggl(\frac{3-\gamma}{2}\biggl)-(\gamma-1)\int_{0}^{\infty}e^{-\frac{y^{2}}{4\mu t}}y(y+\theta_{1})^{-\gamma}dy\biggl), 
\end{split}
\end{align}
where $\Gamma(s)$ is the Gamma function for $s>0$. On the other hand, by making the integration by parts, we have from \eqref{4-19} and \eqref{4-20}
\begin{align}
\label{4-22}
\begin{split}
L_{2}(t)&=\frac{b(c_{\alpha, \beta}^{+}+c_{\alpha, \beta}^{-})}{4\sqrt{\pi}\mu^{3/2}}\eta(at, t)\chi(at, t)t^{-\frac{1}{2}} \\
&\ \ \ \ \times \biggl(\frac{1}{\gamma-2}+\frac{1}{2\mu(2-\gamma)}t^{-1}\int_{0}^{\infty}e^{-\frac{y^{2}}{4\mu t}}y(1+y)^{-(\gamma-2)}dy\biggl) \\
&=\frac{b(c_{\alpha, \beta}^{+}+c_{\alpha, \beta}^{-})}{4\sqrt{\pi}\mu^{3/2}}\eta(at, t)\chi(at, t)t^{-\frac{1}{2}}\biggl(\frac{1}{\gamma-2}+\frac{1}{2\mu(2-\gamma)}t^{-1}\\
&\ \ \ \ \times \biggl(2^{3-\gamma}(\mu t)^{2-\frac{\gamma}{2}}\Gamma \biggl(2-\frac{\gamma}{2}\biggl)+(2-\gamma)\int_{0}^{\infty}e^{-\frac{y^{2}}{4\mu t}}y(y+\theta_{2})^{-(\gamma-1)}dy\biggl)\biggl). \\
\end{split}
\end{align}
Therefore, from \eqref{4-18}, \eqref{4-21}, \eqref{4-22} and \eqref{2-21}, we obtain 
\begin{align}
\label{4-23}
\begin{split}
&\|Z(\cdot, t)\|_{L^{\infty}}\ge|Z(at, t)|=|L_{1}(t)+L_{2}(t)| \\
&=\frac{|\eta(at, t)|}{4\sqrt{\pi}\mu^{3/2}}
\biggl| 2^{2-\gamma}\mu^{\frac{3-\gamma}{2}}(c_{\alpha, \beta}^{+}-c_{\alpha, \beta}^{-})\Gamma\biggl(\frac{3-\gamma}{2}\biggl)t^{-\frac{\gamma}{2}}\\
&\ \ \ \ -(c_{\alpha, \beta}^{+}-c_{\alpha, \beta}^{-})(\gamma-1)t^{-\frac{3}{2}}\int_{0}^{\infty}e^{-\frac{y^{2}}{4\mu t}}y(y+\theta_{1})^{-\gamma}dy-\frac{b(c_{\alpha, \beta}^{+}+c_{\alpha, \beta}^{-})}{2-\gamma}\chi(at, t)t^{-\frac{1}{2}}\\
&\ \ \ \ +\frac{2^{2-\gamma}\mu^{1-\frac{\gamma}{2}}b(c_{\alpha, \beta}^{+}+c_{\alpha, \beta}^{-})}{2-\gamma}\Gamma\biggl(2-\frac{\gamma}{2}\biggl)\chi(at, t)t^{\frac{1-\gamma}{2}}\\
&\ \ \ \ +\frac{b(c_{\alpha, \beta}^{+}+c_{\alpha, \beta}^{-})}{2\mu}\chi(at, t)t^{-\frac{3}{2}}\int_{0}^{\infty}e^{-\frac{y^{2}}{4\mu t}}y(y+\theta_{2})^{-(\gamma-1)}dy\biggl| \\
&\ge \frac{\min\{1, e^{\frac{bM}{2\mu}}\}}{4\sqrt{\pi}\mu^{3/2}}|M_{1}(t)+M_{2}(t)+M_{3}(t)+M_{4}(t)+M_{5}(t)|. 
\end{split}
\end{align}
First, we focus on the term $M_{1}(t)$ and $M_{4}(t)$. From the mean value theorem, there exists $\tilde{\theta} \in (0, 1)$ such that 
\begin{equation*}
\chi_{*}\biggl(\frac{-a}{\sqrt{1+t}}\biggl)=\chi_{*}(0)-\frac{a}{\sqrt{1+t}}\chi_{*}'\biggl(\frac{-a\tilde{\theta}}{\sqrt{1+t}}\biggl).
\end{equation*}
Therefore, we get
\begin{equation}
\label{4-24}
\chi(at, t)=(1+t)^{-\frac{1}{2}}\chi_{*}(0)-a(1+t)^{-1}\chi_{*}'\biggl(\frac{-a\tilde{\theta}}{\sqrt{1+t}}\biggl).
\end{equation}
Thus, using \eqref{4-24}, we have 
\begin{align}
\label{4-25}
\begin{split}
&|M_{1}(t)+M_{4}(t)|\\
&=2^{2-\gamma}\mu^{1-\frac{\gamma}{2}}\biggl|\sqrt{\mu}(c_{\alpha, \beta}^{+}-c_{\alpha, \beta}^{-})\Gamma\biggl(\frac{3-\gamma}{2}\biggl)t^{-\frac{\gamma}{2}}+\frac{b(c_{\alpha, \beta}^{+}+c_{\alpha, \beta}^{-})}{2-\gamma}\Gamma\biggl(2-\frac{\gamma}{2}\biggl)\chi(at, t)t^{\frac{1-\gamma}{2}}\biggl|\\
&=2^{2-\gamma}\mu^{1-\frac{\gamma}{2}}\biggl|\sqrt{\mu}(c_{\alpha, \beta}^{+}-c_{\alpha, \beta}^{-})\Gamma\biggl(\frac{3-\gamma}{2}\biggl)t^{-\frac{\gamma}{2}}+\frac{b\chi_{*}(0)(c_{\alpha, \beta}^{+}+c_{\alpha, \beta}^{-})}{2-\gamma}\Gamma\biggl(2-\frac{\gamma}{2}\biggl)(1+t)^{-\frac{1}{2}}t^{\frac{1-\gamma}{2}}\\
&\ \ \ -\frac{ab(c_{\alpha, \beta}^{+}+c_{\alpha, \beta}^{-})}{2-\gamma}\Gamma\biggl(2-\frac{\gamma}{2}\biggl)(1+t)^{-1}t^{\frac{1-\gamma}{2}}\chi_{*}'\biggl(\frac{-a\tilde{\theta}}{\sqrt{1+t}}\biggl)\biggl|\\
&=2^{2-\gamma}\mu^{1-\frac{\gamma}{2}}\biggl|\biggl(\sqrt{\mu}(c_{\alpha, \beta}^{+}-c_{\alpha, \beta}^{-})\Gamma\biggl(\frac{3-\gamma}{2}\biggl)+\frac{b\chi_{*}(0)(c_{\alpha, \beta}^{+}+c_{\alpha, \beta}^{-})}{2-\gamma}\Gamma\biggl(2-\frac{\gamma}{2}\biggl)\biggl)t^{-\frac{\gamma}{2}}\\
&\ \ \ +\frac{b\chi_{*}(0)(c_{\alpha, \beta}^{+}+c_{\alpha, \beta}^{-})}{2-\gamma}\Gamma\biggl(2-\frac{\gamma}{2}\biggl)((1+t)^{-\frac{1}{2}}-t^{-\frac{1}{2}})t^{\frac{1-\gamma}{2}}\\
&\ \ \ -\frac{ab(c_{\alpha, \beta}^{+}+c_{\alpha, \beta}^{-})}{2-\gamma}\Gamma\biggl(2-\frac{\gamma}{2}\biggl)(1+t)^{-1}t^{\frac{1-\gamma}{2}}\chi_{*}'\biggl(\frac{-a\tilde{\theta}}{\sqrt{1+t}}\biggl)\biggl|.
\end{split}
\end{align}
From \eqref{1-4} and \eqref{1-6}, it is easy to see that 
\begin{equation*}
\chi_{*}'+\frac{x}{2\mu}\chi_{*}=\frac{b}{2\mu}\chi_{*}^{2}.
\end{equation*}
Then, using \eqref{2-11}, we obtain 
\begin{equation*}
|\chi_{*}'|\le C(|x|+|\chi_{*}|)|\chi_{*}|\le C|M|e^{-\frac{x^{2}}{8\mu}},
\end{equation*}
and hence 
\[\biggl|\chi_{*}'\biggl(\frac{-a\tilde{\theta}}{\sqrt{1+t}}\biggl)\biggl|\le C|M|e^{-\frac{\tilde{\theta}^{2}a^{2}}{8\mu(1+t)}}\le C=:C_{0}.\]
Therefore, from \eqref{4-25}, we get 
\begin{align}
\label{4-26}
\begin{split}
|M_{1}(t)+M_{4}(t)|&\ge 2^{2-\gamma}\mu^{1-\frac{\gamma}{2}}\biggl(|\tilde{\nu_{0}}|(1+t)^{-\frac{\gamma}{2}}-\frac{|b\chi_{*}(0)(c_{\alpha, \beta}^{+}+c_{\alpha, \beta}^{-})|}{2-\gamma}\Gamma\biggl(2-\frac{\gamma}{2}\biggl)t^{-1-\frac{\gamma}{2}}\\
&\ \ \ -\frac{C_{0}|ab(c_{\alpha, \beta}^{+}+c_{\alpha, \beta}^{-})|}{2-\gamma}\Gamma\biggl(2-\frac{\gamma}{2}\biggl)t^{-\frac{1+\gamma}{2}}\biggl),
\end{split}
\end{align}
where $\tilde{\nu_{0}}$ is defined by \eqref{1-25}. Next, for $M_{2}(t)$, we have from \eqref{4-20}
\begin{align}
\label{4-27}
\begin{split}
|M_{2}(t)|&\le (\gamma-1)|c_{\alpha, \beta}^{+}-c_{\alpha, \beta}^{-}|t^{-\frac{3}{2}}\int_{0}^{\infty}e^{-\frac{y^{2}}{4\mu t}}y^{1-\gamma}dy\\
&=\frac{\mu^{1-\frac{\gamma}{2}}(\gamma-1)|c_{\alpha, \beta}^{+}-c_{\alpha, \beta}^{-}|}{2^{\gamma-1}}\Gamma \biggl(1-\frac{\gamma}{2}\biggl)t^{-\frac{1+\gamma}{2}}.\\
\end{split}
\end{align}
It is easy to see that 
\begin{align}
\label{4-28}
\begin{split}
|M_{3}(t)|&\le \frac{|b(c_{\alpha, \beta}^{+}+c_{\alpha, \beta}^{-})|}{2-\gamma}\|\chi(\cdot, t)\|_{L^{\infty}}t^{-\frac{1}{2}}\le \frac{C_{1}|b(c_{\alpha, \beta}^{+}+c_{\alpha, \beta}^{-})|}{2-\gamma}t^{-1}.
\end{split}
\end{align}
Finally, for $M_{5}(t)$, we obtain from \eqref{4-20}
\begin{align}
\label{4-29}
\begin{split}
|M_{5}(t)|&\le\frac{|b(c_{\alpha, \beta}^{+}+c_{\alpha, \beta}^{-})|}{2\mu}
\|\chi(\cdot, t)\|_{L^{\infty}}t^{-\frac{3}{2}}\int_{0}^{\infty}e^{-\frac{y^{2}}{4\mu t}}y^{2-\gamma}dy\\
&\le \frac{C_{2}|b(c_{\alpha, \beta}^{+}+c_{\alpha, \beta}^{-})|}{(2\sqrt{\mu})^{\gamma-1}}\Gamma\biggl(\frac{3-\gamma}{2}\biggl)t^{-\frac{1+\gamma}{2}}.
\end{split}
\end{align}
Therefore, combining \eqref{4-23}, \eqref{4-26} through \eqref{4-29}, we have
\begin{align*}
\|Z(\cdot, t)\|_{L^{\infty}}&\ge \frac{\min\{1, e^{\frac{bM}{2\mu}}\}}{4\sqrt{\pi}\mu^{3/2}}(|M_{1}(t)+M_{4}(t)|-|M_{2}(t)|-|M_{3}(t)|-|M_{5}(t)|) \\
&\ge \frac{\min\{1, e^{\frac{bM}{2\mu}}\}}{4\sqrt{\pi}\mu^{3/2}}\biggl\{2^{2-\gamma}\mu^{1-\frac{\gamma}{2}}\biggl(|\tilde{\nu_{0}}|(1+t)^{-\frac{\gamma}{2}}-\frac{|b\chi_{*}(0)(c_{\alpha, \beta}^{+}+c_{\alpha, \beta}^{-})|}{2-\gamma}\Gamma\biggl(2-\frac{\gamma}{2}\biggl)t^{-1-\frac{\gamma}{2}}\\
&\ \ \ -\frac{C_{0}|ab(c_{\alpha, \beta}^{+}+c_{\alpha, \beta}^{-})|}{2-\gamma}\Gamma\biggl(2-\frac{\gamma}{2}\biggl)t^{-\frac{1+\gamma}{2}}\biggl)-\frac{\mu^{1-\frac{\gamma}{2}}(\gamma-1)|c_{\alpha, \beta}^{+}-c_{\alpha, \beta}^{-}|}{2^{\gamma-1}}\Gamma \biggl(1-\frac{\gamma}{2}\biggl)t^{-\frac{1+\gamma}{2}}\\
&\ \ \ -\frac{C_{1}|b(c_{\alpha, \beta}^{+}+c_{\alpha, \beta}^{-})|}{2-\gamma}t^{-1}-\frac{C_{2}|b(c_{\alpha, \beta}^{+}+c_{\alpha, \beta}^{-})|}{(2\sqrt{\mu})^{\gamma-1}}\Gamma\biggl(\frac{3-\gamma}{2}\biggl)t^{-\frac{1+\gamma}{2}}\biggl\}.
\end{align*}
Hence, there is a positive constant $\nu_{0}$ such that \eqref{1-23} holds. This completes the proof of Theorem~\ref{main2} for $1<\min\{\alpha, \beta\}<2$. 
\end{proof}

\section{Proof of Theorem~\ref{main2} for $\min\{\alpha, \beta\}=2$}

Finally in this section, we shall completes the proof of Theorem~\ref{main2}. Namely, we prove \eqref{1-20} and \eqref{1-24}. First, let us recall the following fact derived in~\cite{Ka07}. We consider 
\begin{align}
\label{5-1}
\begin{split}
v_{t}+av_{x}+(b\chi v)_{x}-\mu v_{xx}&=-\kappa\p_{x}(\chi^{3}), \ x\in \R, \ t>0, \\
v(x, 0)&=0, \ x\in \R.
\end{split}
\end{align}
The leading term of the solution $v(x, t)$ to \eqref{5-1} is given by $V(x, t)$ defined by \eqref{1-9}. More precisely, the following asymptotic formula can be shown (for the proof, see Proposition 4.3 in~\cite{Ka07}). 
\begin{prop}\label{AF2}
Assume that $|M| \le 1$. Then the estimate 
\begin{equation}
\label{5-2}
\|v(\cdot, t)-V(\cdot, t)\|_{L^{\infty}} \le C|M|(1+t)^{-1}, \ t\ge1
\end{equation}
holds. Here $v(x, t)$ is the solution to \eqref{5-1} and $V(x, t)$ is defined by \eqref{1-9}.
\end{prop}
\noindent
This formula helps us to complete the proof of Theorem~\ref{main2} with $\min\{\alpha, \beta\}=2$.

\medskip
\begin{proof}[\rm{\bf{End of the proof of Theorem~\ref{main2} for $\min\{\alpha, \beta\}=2$}}]
First, we shall prove \eqref{1-20}. We set 
\begin{align*}
\phi(x, t)&\equiv u(x, t)+u_{t}(x, t)-\chi(x, t)-v(x, t).
\end{align*}
Then, from \eqref{1-2}, \eqref{1-6} and \eqref{5-1}, we have the following initial value problem:
\begin{align}
\label{5-3}
\begin{split}
\phi_{t}+a\phi_{x}+(b\chi \phi)_{x}-\mu \phi_{xx}&=\p_{x}N_{1}(\chi)+\p_{x}N_{2}(u, \chi), \ x\in \R, \ t>0, \\
\phi(x, 0)&=\psi_{0}(x)=u_{0}(x)+u_{1}(x)-\chi_{0}(x), \ x\in \R,
\end{split}
\end{align}
where 
\begin{align}
\label{5-4}
\begin{split}
N_{1}(\chi)&\equiv 2a\mu\chi_{xx}-2ab\chi\chi_{x}+\frac{ab^{2}}{4\mu}\chi^{3},\\
N_{2}(u, \chi)&\equiv a(\p_{t}+a\p_{x})(u-\chi)-\mu \p_{t}\p_{x}(u-\chi)+b\chi\p_{t}(u-\chi)\\
&\ \ \ -\mu\p_{x}(\p_{t}+a\p_{x})\chi+b\chi(\p_{t}+a\p_{x})\chi-\frac{b}{2}(u-\chi)^{2}-\frac{c}{3!}(u-\chi)^{3}-\frac{c}{2}u\chi(u-\chi).
\end{split}
\end{align}
Therefore, from Lemma~\ref{RF}, we obtain 
\begin{align}
\label{5-5}
\begin{split}
\phi(x, t)=&\ U[\psi_{0}](x, t, 0)+\int_{0}^{t}U[\p_{x}N_{1}(\chi)(\tau)](x, t, \tau)d\tau+\int_{0}^{t}U[\p_{x}N_{2}(u, \chi)(\tau)](x, t, \tau)d\tau.
\end{split}
\end{align}
Thus, we have 
\begin{align}
\label{5-6}
\begin{split}
&u(x, t)-\chi(x, t)-Z(x, t)-V(x, t)\\
&=U[\psi_{0}](x, t, 0)-Z(x, t)-u_{t}(x, t)+v(x, t)-V(x, t)\\
&\ \ \ +\int_{0}^{t}U[\p_{x}N_{1}(\chi)(\tau)](x, t, \tau)d\tau+\int_{0}^{t}U[\p_{x}N_{2}(u, \chi)(\tau)](x, t, \tau)d\tau\\
&\equiv U[\psi_{0}](x, t, 0)-Z(x, t)-u_{t}(x, t)+v(x, t)-V(x, t)+K_{1}+K_{2},
\end{split}
\end{align}
where $Z(x, t)$ and $V(x, t)$ are defined by \eqref{1-21} and \eqref{1-9}, respectively. 

Now, we only need to evaluate the last two terms in the right hand side of \eqref{5-6}. First, we evaluate $K_{1}$. To estimate it, we introduce the useful property of $N_{1}(\chi)$. Actually, if we set $N_{0}(\chi)\equiv 2\mu\chi_{x}-\frac{b}{2}\chi^{2}$, from $N_{1}(\chi)=a(\p_{x} N_{0}(\chi)-\frac{b}{2\mu}\chi N_{0}(\chi))$, we get $N_{1}(\chi)=\eta\p_{x}(\eta^{-1}N_{0}(\chi))$. Therefore, from the definition of $K_{1}$ and \eqref{2-26}, and by making the integration by parts, we have
\begin{align}
\label{5-7}
\begin{split}
K_{1}(x, t)&=\int_{0}^{t}\int_{\R}\p_{x}(G_{0}(x-y, t-\tau)\eta(x, t))\p_{y}((\eta(y, \tau))^{-1}N_{0}(y, \tau))dyd\tau\\
&=\sum_{n=0}^{1}\p_{x}^{1-n}\eta(x, t)\biggl(\int_{0}^{t/2}+\int_{t/2}^{t}\biggl)\int_{\R}\p_{x}^{n}G_{0}(x-y, t-\tau)\p_{y}((\eta(y, \tau))^{-1}N_{0}(y, \tau))dyd\tau\\
&=\sum_{n=0}^{1}\p_{x}^{1-n}\eta(x, t)\biggl(\int_{0}^{t/2}\int_{\R}\p_{x}^{n+1}G_{0}(x-y, t-\tau)(\eta(y, \tau))^{-1}N_{0}(y, \tau)dyd\tau\\
&\ \ \ +\int_{t/2}^{t}\int_{\R}\p_{x}^{n}G_{0}(x-y, t-\tau)\p_{y}((\eta(y, \tau))^{-1}N_{0}(y, \tau))dyd\tau\biggl).
\end{split}
\end{align}
Also from Lemma~\ref{chi.decay} and Lemma~\ref{eta.decay}, for any non-negative integer $l$ and $1\le q\le \infty$, it is easy to see that
\begin{equation}
\label{5-8}
\|\p_{x}^{l}(\eta^{-1}N_{0}(\chi))(\cdot, t)\|_{L^{q}}\le C\sum_{j=0}^{l}(1+t)^{-\frac{1}{2}(l-j)}\|\p_{x}^{j}N_{0}(\cdot, t)\|_{L^{q}}\le C(1+t)^{-1+\frac{1}{2q}-\frac{l}{2}}.
\end{equation} 
Hence, from \eqref{5-7}, Young's inequality, Lemma~\ref{eta.decay}, \eqref{2-13} and \eqref{5-8}, we have 
\begin{align}
\label{5-9}
\begin{split}
\|K_{1}(\cdot, t)\|_{L^{\infty}}&\le C\sum_{n=0}^{1}\|\p_{x}^{1-n}\eta(\cdot, t)\|_{L^{\infty}}\biggl(\int_{0}^{t/2}\|\p_{x}^{n+1}G_{0}(\cdot, t-\tau)\|_{L^{\infty}}\|(\eta^{-1}N_{0}(\chi))(\cdot, \tau)\|_{L^{1}}d\tau\\
&\ \ \ +\int_{t/2}^{t}\|\p_{x}^{n}G_{0}(\cdot, t-\tau)\|_{L^{1}}\|\p_{x}(\eta^{-1}N_{0}(\chi))(\cdot, \tau)\|_{L^{\infty}}d\tau\biggl)\\
&\le C\sum_{n=0}^{1}(1+t)^{-\frac{1}{2}+\frac{n}{2}}\biggl(\int_{0}^{t/2}(t-\tau)^{-1-\frac{n}{2}}(1+\tau)^{-\frac{1}{2}}d\tau+\int_{t/2}^{t}(t-\tau)^{-\frac{n}{2}}(1+\tau)^{-\frac{3}{2}}d\tau\biggl)\\
&\le C(1+t)^{-1}, \ t\ge1. 
\end{split}
\end{align}
Next, we estimate $K_{2}$. Before do that, for $0<\e<\frac{1}{2}$, we prepare the following estimates:
\begin{align}
\label{5-10}
\|N_{2}(\cdot, t)\|_{L^{1}}&\le C(1+t)^{-\frac{3}{2}+\e},\\
\label{5-11}
\|N_{2}(\cdot, t)\|_{L^{2}}&\le C(1+t)^{-\frac{7}{4}+\e}.
\end{align}
We shall prove only \eqref{5-10}, since we can prove \eqref{5-11} in the same way. From \eqref{3-57}, \eqref{1-17}, \eqref{2-12}, \eqref{1-16} and \eqref{2-9}, we have 
\begin{align*}
\|N_{2}(\cdot, t)\|_{L^{1}}&\le C\|(\p_{t}+a\p_{x})(u-\chi)(\cdot, t)\|_{L^{1}}+C\|\p_{t}\p_{x}(u-\chi)(\cdot, t)\|_{L^{1}}\\
&\ \ \ +C\|\chi(\cdot, t)\|_{L^{\infty}}\|\p_{t}(u-\chi)(\cdot, t)\|_{L^{1}}+C\|\p_{x}(\p_{t}+a\p_{x})\chi(\cdot, t)\|_{L^{1}}\\
&\ \ \ +C\|\chi(\cdot, t)\|_{\infty}\|(\p_{t}+a\p_{x})\chi(\cdot, t)\|_{L^{1}}+C\|(u-\chi)^{2}(\cdot, t)\|_{L^{1}}\\
&\ \ \ +C\|(u-\chi)^{3}(\cdot, t)\|_{L^{1}}+C\|u(\cdot, t)\|_{L^{\infty}}\|\chi(\cdot, t)\|_{L^{\infty}}\|(u-\chi)(\cdot, t)\|_{L^{1}}\\
&\le C(1+t)^{-\frac{3}{2}+\e}. 
\end{align*}
Therefore, by using Lemma~\ref{N.decay}, \eqref{5-10} and \eqref{5-11}, we obtain 
\begin{align}
\label{5-12}
\begin{split}
&\|K_{2}(\cdot, t)\|_{L^{\infty}}\\
&\le C\sum_{n=0}^{1}(1+t)^{-\frac{1}{2}+\frac{n}{2}}\biggl(\int_{0}^{t/2}(t-\tau)^{-\frac{1}{2}-\frac{n}{2}}\| N_{2}(\cdot, \tau)\|_{L^{1}}d\tau+\int_{t/2}^{t}(t-\tau)^{-\frac{1}{4}-\frac{n}{2}}\| N_{2}(\cdot, \tau)\|_{L^{2}}d\tau \biggl) \\
&\le C\sum_{n=0}^{1}(1+t)^{-\frac{1}{2}+\frac{n}{2}}\biggl(\int_{0}^{t/2}(t-\tau)^{-\frac{1}{2}-\frac{n}{2}}(1+\tau)^{-\frac{3}{2}+\e}d\tau+\int_{t/2}^{t}(t-\tau)^{-\frac{1}{4}-\frac{n}{2}}(1+\tau)^{-\frac{7}{4}+\e}d\tau \biggl) \\
&\le C(1+t)^{-1}, \ t\ge1. 
\end{split}
\end{align}
Thus, from \eqref{5-6}, \eqref{2-10}, Gagliardo-Nirenberg inequality, \eqref{5-2}, \eqref{5-9} and \eqref{5-12}, we obtain 
\begin{align*}
\|u(\cdot, t)-\chi(\cdot, t)-Z(\cdot, t)-V(\cdot, t)\|_{L^{\infty}}\le \|U[\psi_{0}](\cdot, t, 0)-Z(\cdot, t)\|_{L^{\infty}}+C(1+t)^{-1}, \ \ t\ge1.
\end{align*}
Therefore, from \eqref{4-5}, we finally arrive at 
\begin{equation*}
\limsup_{t\to \infty}\frac{(1+t)}{\log(1+t)}\|u(\cdot, t)-\chi(\cdot, t)-Z(\cdot, t)-V(\cdot, t)\|_{L^{\infty}}=0.
\end{equation*}
This completes the proof of \eqref{1-20}.

Finally, we shall derive the estimate of $Z(x, t)+V(x, t)$, that is \eqref{1-24}. Since 
\[\|V(\cdot, t)\|_{L^{\infty}}=|\kappa d|\|V_{*}(\cdot)\|_{L^{\infty}}(1+t)^{-1}\log(1+t),\]
and using \eqref{4-17}, we can derive the upper bound estimate of \eqref{1-24}. Then, we shall only prove the lower bound estimate. First, we take $x=at$ in \eqref{1-9}. Since 
\begin{align*}
V_{*}(x)&=\frac{1}{\sqrt{4\pi \mu}}\biggl(\frac{b}{2\mu}\chi_{*}(x)\eta_{*}(x)-\frac{x}{2\mu}\eta_{*}(x)\biggl)=\frac{1}{4\sqrt{\pi}\mu^{3/2}}(b\chi_{*}(x)-x)\eta_{*}(x),
\end{align*}
from \eqref{1-10}, \eqref{1-4} and \eqref{1-22}, it follows that 
\begin{align}
\label{5-13}
\begin{split}
V(at, t)&=-\kappa dV_{*}\biggl(\frac{-a}{\sqrt{1+t}}\biggl)(1+t)^{-1}\log(1+t)\\
&=-\frac{\kappa d}{4\sqrt{\pi}\mu^{3/2}}\biggl(b\sqrt{1+t}\chi(at, t)+\frac{a}{\sqrt{1+t}}\biggl)\eta(at, t)(1+t)^{-1}\log(1+t)\\
&=-\frac{\kappa d}{4\sqrt{\pi}\mu^{3/2}}\eta(at, t)\biggl(b\chi(at, t)(1+t)^{-\frac{1}{2}}\log(1+t)+a(1+t)^{-\frac{3}{2}}\log(1+t)\biggl).
\end{split}
\end{align}
Combining \eqref{4-18} and \eqref{5-13}, we have from the triangle inequality that 
\begin{align}
\label{5-14}
\begin{split}
&\|Z(\cdot, t)+V(\cdot, t)\|_{L^{\infty}}\ge |Z(at, t)+V(at, t)|\\
&=\biggl|\frac{(c_{\alpha, \beta}^{+}-c_{\alpha, \beta}^{-})}{4\sqrt{\pi}\mu^{3/2}}\eta(at, t)t^{-\frac{3}{2}}\int_{0}^{\infty}e^{-\frac{y^{2}}{4\mu t}}y(1+y)^{-1}dy\\
&\ \ \ \ +\frac{b(c_{\alpha, \beta}^{+}+c_{\alpha, \beta}^{-})}{4\sqrt{\pi}\mu^{3/2}}\eta(at, t)\chi(at, t)t^{-\frac{1}{2}}\int_{0}^{\infty}e^{-\frac{y^{2}}{4\mu t}}(1+y)^{-1}dy\\
&\ \ \ \ -\frac{\kappa d}{4\sqrt{\pi}\mu^{3/2}}\eta(at, t)\biggl(b\chi(at, t)(1+t)^{-\frac{1}{2}}\log(1+t)+a(1+t)^{-\frac{3}{2}}\log(1+t)\biggl)\biggl|\\
&\ge \frac{|b|\min\{1, e^{\frac{bM}{2\mu}}\}}{4\sqrt{\pi}\mu^{3/2}}|\chi(at, t)|\bigg|(c_{\alpha, \beta}^{+}+c_{\alpha, \beta}^{-})t^{-\frac{1}{2}}\int_{0}^{\infty}e^{-\frac{y^{2}}{4\mu t}}(1+y)^{-1}dy\\
&\ \ \ \ -\kappa d(1+t)^{-\frac{1}{2}}\log(1+t)\biggl|-\frac{|(c_{\alpha, \beta}^{+}-c_{\alpha, \beta}^{-})|\max\{1, e^{\frac{bM}{2\mu}}\}}{4\sqrt{\pi}\mu^{3/2}}t^{-\frac{3}{2}}\int_{0}^{\infty}e^{-\frac{y^{2}}{4\mu t}}dy\\
&\ \ \ \ -\frac{|a\kappa d|\max\{1, e^{\frac{bM}{2\mu}}\}}{4\sqrt{\pi}\mu^{3/2}}(1+t)^{-\frac{3}{2}}\log(1+t)\\
&\equiv \frac{|b|\min\{1, e^{\frac{bM}{2\mu}}\}}{4\sqrt{\pi}\mu^{3/2}}|\chi(at, t)|W(t)-\frac{|(c_{\alpha, \beta}^{+}-c_{\alpha, \beta}^{-})|\max\{1, e^{\frac{bM}{2\mu}}\}}{4\mu}t^{-1}\\
&\ \ \ \ -\frac{|a\kappa d|\max\{1, e^{\frac{bM}{2\mu}}\}}{4\sqrt{\pi}\mu^{3/2}}(1+t)^{-\frac{3}{2}}\log(1+t).\\
\end{split}
\end{align}
Now, we evaluate $W(t)$ from below. Splitting the $y$-integral and using the triangle inequality, we obtain 
\begin{align}
\label{5-15}
\begin{split}
W(t)=&\ \biggl|(c_{\alpha, \beta}^{+}+c_{\alpha, \beta}^{-})t^{-\frac{1}{2}}\biggl(\int_{0}^{\sqrt{1+t}-1}+\int_{\sqrt{1+t}-1}^{\infty}\biggl)e^{-\frac{y^{2}}{4\mu t}}(1+y)^{-1}dy-\kappa d(1+t)^{-\frac{1}{2}}\log(1+t)\biggl|\\
\ge&\ \biggl|(c_{\alpha, \beta}^{+}+c_{\alpha, \beta}^{-})t^{-\frac{1}{2}}\int_{0}^{\sqrt{1+t}-1}e^{-\frac{y^{2}}{4\mu t}}(1+y)^{-1}dy-\kappa d(1+t)^{-\frac{1}{2}}\log(1+t)\biggl|\\
&\ -\biggl|(c_{\alpha, \beta}^{+}+c_{\alpha, \beta}^{-})t^{-\frac{1}{2}}\int_{\sqrt{1+t}-1}^{\infty}e^{-\frac{y^{2}}{4\mu t}}(1+y)^{-1}dy\biggl|\\
\equiv&\ W_{1}(t)-W_{2}(t).
\end{split}
\end{align}
From the mean value theorem, there exists $\theta \in(0, 1)$ such that 
\begin{equation*}
e^{-\frac{y^{2}}{4\mu t}}=1-\frac{y^{2}}{4\mu t}e^{-\theta \frac{y^{2}}{4\mu t}}. 
\end{equation*}
Therefore, we have  
\begin{align}
\label{5-16}
\begin{split}
W_{1}(t)=&\ \biggl|(c_{\alpha, \beta}^{+}+c_{\alpha, \beta}^{-})t^{-\frac{1}{2}}\int_{0}^{\sqrt{1+t}-1}\biggl(1-\frac{y^{2}}{4\mu t}e^{-\theta \frac{y^{2}}{4\mu t}}\biggl)(1+y)^{-1}dy-\kappa d(1+t)^{-\frac{1}{2}}\log(1+t)\biggl|\\
\ge&\ \biggl|(c_{\alpha, \beta}^{+}+c_{\alpha, \beta}^{-})t^{-\frac{1}{2}}\int_{0}^{\sqrt{1+t}-1}(1+y)^{-1}dy-\kappa d(1+t)^{-\frac{1}{2}}\log(1+t)\biggl|\\
&\ -\biggl|\frac{c_{\alpha, \beta}^{+}+c_{\alpha, \beta}^{-}}{4\mu}t^{-\frac{3}{2}}\int_{0}^{\sqrt{1+t}-1}e^{-\theta \frac{y^{2}}{4\mu t}}y^{2}(1+y)^{-1}dy\biggl|\\
\equiv&\ W_{1.1}(t)-W_{1.2}(t).
\end{split}
\end{align}
For $W_{1.1}(t)$, we obtain  
\begin{align}
\label{5-17}
\begin{split}
W_{1.1}(t)=&\ \biggl|\frac{c_{\alpha, \beta}^{+}+c_{\alpha, \beta}^{-}}{2}t^{-\frac{1}{2}}\log(1+t)-\kappa d(1+t)^{-\frac{1}{2}}\log(1+t)\biggl|\\
=&\ \biggl|\frac{c_{\alpha, \beta}^{+}+c_{\alpha, \beta}^{-}}{2}t^{-\frac{1}{2}}\log(1+t)-\frac{c_{\alpha, \beta}^{+}+c_{\alpha, \beta}^{-}}{2}(1+t)^{-\frac{1}{2}}\log(1+t)\\
&\ +\biggl(\frac{c_{\alpha, \beta}^{+}+c_{\alpha, \beta}^{-}}{2}-\kappa d\biggl)(1+t)^{-\frac{1}{2}}\log(1+t)\biggl| \\
\ge&\ |\tilde{\nu_{1}}|(1+t)^{-\frac{1}{2}}\log(1+t)-\frac{|c_{\alpha, \beta}^{+}+c_{\alpha, \beta}^{-}|}{2}\log(1+t)(t^{-\frac{1}{2}}-(1+t)^{-\frac{1}{2}})\\
\ge&\ |\tilde{\nu_{1}}|(1+t)^{-\frac{1}{2}}\log(1+t)-\frac{|c_{\alpha, \beta}^{+}+c_{\alpha, \beta}^{-}|}{2}t^{-\frac{3}{2}}\log(1+t),
\end{split}
\end{align}
where $\tilde{\nu_{1}}$ is defined by \eqref{1-25}. On the other hand, for $W_{1.2}(t)$, we get 
\begin{align}
\label{5-18}
\begin{split}
W_{1.2}(t)&\le \frac{|c_{\alpha, \beta}^{+}+c_{\alpha, \beta}^{-}|}{4\mu}t^{-\frac{3}{2}}\int_{0}^{\sqrt{1+t}-1}y\ dy=\frac{|c_{\alpha, \beta}^{+}+c_{\alpha, \beta}^{-}|}{8\mu}t^{-\frac{3}{2}}(\sqrt{1+t}-1)^{2}\le \frac{|c_{\alpha, \beta}^{+}+c_{\alpha, \beta}^{-}|}{8\mu}t^{-\frac{1}{2}}. 
\end{split}
\end{align}
Analogously, for $W_{2}(t)$, it is easy to see that 
\begin{align}
\label{5-19}
\begin{split}
W_{2}(t)&\le|c_{\alpha, \beta}^{+}+c_{\alpha, \beta}^{-}|t^{-\frac{1}{2}}\biggl(\sup_{y\ge \sqrt{1+t}-1}(1+|y|)^{-1}\biggl)\int_{y\ge \sqrt{1+t}-1}e^{-\frac{y^{2}}{4\mu t}}dy\\
&\le \sqrt{\mu \pi}|c_{\alpha, \beta}^{+}+c_{\alpha, \beta}^{-}|(1+t)^{-\frac{1}{2}}.
\end{split}
\end{align}
Finally, combining \eqref{5-14} through \eqref{5-19}, \eqref{2-12} and \eqref{4-24}, we obtain 
\begin{align*}
&\|Z(\cdot, t)+V(\cdot, t)\|_{L^{\infty}}\\
&\ge \frac{|b|\min\{1, e^{\frac{bM}{2\mu}}\}}{4\sqrt{\pi}\mu^{3/2}}|\chi(at, t)||\tilde{\nu_{1}}|(1+t)^{-\frac{1}{2}}\log(1+t)\\
&\ \ \ \ -\frac{|b|\min\{1, e^{\frac{bM}{2\mu}}\}|c_{\alpha, \beta}^{+}+c_{\alpha, \beta}^{-}|}{8\sqrt{\pi}\mu^{3/2}}t^{-\frac{3}{2}}\log(1+t)\|\chi(\cdot, t)\|_{L^{\infty}}\\
&\ \ \ \ -\frac{|b|\min\{1, e^{\frac{bM}{2\mu}}\}|c_{\alpha, \beta}^{+}+c_{\alpha, \beta}^{-}|}{4\sqrt{\pi}\mu^{3/2}}\biggl(\sqrt{\mu \pi}+\frac{1}{8\mu}\biggl)t^{-\frac{1}{2}}\|\chi(\cdot, t)\|_{L^{\infty}}\\
&\ \ \ \ -\frac{|(c_{\alpha, \beta}^{+}-c_{\alpha, \beta}^{-})|\max\{1, e^{\frac{bM}{2\mu}}\}}{4\mu}t^{-1}-\frac{|a\kappa d|\max\{1, e^{\frac{bM}{2\mu}}\}}{4\sqrt{\pi}\mu^{3/2}}(1+t)^{-\frac{3}{2}}\log(1+t)\\
&\ge \frac{|b\chi_{*}(0)|\min\{1, e^{\frac{bM}{2\mu}}\}}{4\sqrt{\pi}\mu^{3/2}}|\tilde{\nu_{1}}|(1+t)^{-1}\log(1+t)-C_{0}\frac{|ab|\min\{1, e^{\frac{bM}{2\mu}}\}}{4\sqrt{\pi}\mu^{3/2}}|\tilde{\nu_{1}}|(1+t)^{-\frac{3}{2}}\log(1+t)\\
&\ \ \ \ -\frac{|b|\min\{1, e^{\frac{bM}{2\mu}}\}|c_{\alpha, \beta}^{+}+c_{\alpha, \beta}^{-}|}{8\sqrt{\pi}\mu^{3/2}}t^{-2}\log(1+t)\\
&\ \ \ \ -\frac{|b|\min\{1, e^{\frac{bM}{2\mu}}\}|c_{\alpha, \beta}^{+}+c_{\alpha, \beta}^{-}|}{4\sqrt{\pi}\mu^{3/2}}\biggl(\sqrt{\mu \pi}+\frac{1}{8\mu}\biggl)t^{-1}\\
&\ \ \ \ -\frac{|(c_{\alpha, \beta}^{+}-c_{\alpha, \beta}^{-})|\max\{1, e^{\frac{bM}{2\mu}}\}}{4\mu}t^{-1}-\frac{|a\kappa d|\max\{1, e^{\frac{bM}{2\mu}}\}}{4\sqrt{\pi}\mu^{3/2}}(1+t)^{-\frac{3}{2}}\log(1+t).\\
\end{align*}
As a conclusion, there is a positive constant $\nu_{1}$ such that \eqref{1-24} holds. This completes the proof of Theorem~\ref{main2} for $\min\{\alpha, \beta\}=2$. 
\end{proof}

\section*{Acknowledgments}

The author would like to express his sincere gratitude to Professor Hideo Kubo for his feedback and valuable advices. 

\smallskip
The part of this study was done when the author visited School of Mathematical Sciences, Zhejiang University. He would like to thank Professor Jian Zhai for his kindly invitation and grate support. Also, he grateful to Professor Ziheng Tu and Mr Xiaochen Chen for their stimulating discussions.

\smallskip
This study is partially supported by Grant-in-Aid for JSPS Research Fellow No.18J12340 and MEXT through Program for Leading Graduate Schools (Hokkaido University ``Ambitious Leader's Program"). 



\vskip10pt
\par\noindent
\begin{flushleft}Ikki Fukuda\\
Department of Mathematics, Hokkaido University\\
Sapporo 060-0810, Japan\\
E-mail: i.fukuda@math.sci.hokudai.ac.jp
\end{flushleft}

\end{document}